\documentclass[11pt,twoside]{preprint}
\usepackage[full]{textcomp}
\usepackage[osf]{newtxtext}
\usepackage{amsmath,amsthm,amssymb,enumerate}
\usepackage{hyperref}
\usepackage{mathrsfs}
\usepackage{breakurl}
\usepackage{pgfgantt}
\usepackage{footmisc}

\usepackage{xcolor}
\usepackage{mhequ}
\usepackage{comment}
\usepackage{todonotes}

\usepackage{tikz}
\usetikzlibrary{arrows, shapes, snakes}
\usepackage{microtype}
\usepackage[a4paper,innermargin=1.4in,outermargin=1.4in,bottom=1.5in,marginparwidth=1.5in,marginparsep=3mm]{geometry}

\colorlet{darkblue}{blue!90!black}
\colorlet{darkred}{red!90!black}

\newtheorem{theorem}{Theorem}[section]
\newtheorem{lemma}[theorem]{Lemma}
\newtheorem{proposition}[theorem]{Proposition}
\newtheorem{corollary}[theorem]{Corollary}

\theoremstyle{definition}

\newtheorem{definition}[theorem]{Definition}

\newtheorem{remark}[theorem]{Remark}

\def\scal#1{\langle #1 \rangle}

\newcommand{\vn}[1]{{\vert\kern-0.3ex\vert\kern-0.3ex\vert #1 
    \vert\kern-0.3ex\vert\kern-0.3ex\vert}}
    
\newcommand{\wick}[1]{{:\kern-0.3ex #1 \kern-0.3ex:}}

\newcommand{\bn}[1]{{[\kern-0.5ex] #1 
    [\kern-0.5ex]}}

\allowdisplaybreaks

\colorlet{lightred}{red!60!white}

\colorlet{lightgrey}{black!60!white}

\colorlet{darkgreen}{green!70!black}

\newcommand\Z{\mathbb{Z}}

\newcommand\bone{\mathbf{1}}

\newcommand\cD{\mathcal{D}}

\newcommand\cH{\mathcal{H}}

\newcommand\cP{\mathcal{P}}
\newcommand\cF{\mathcal{F}}

\def\eps{\varepsilon}

\newcommand{\R}{\mathbb{R}}
\newcommand{\N}{\mathbb{N}}
\newcommand{\E}{\mathbb{E}}
\newcommand{\T}{\mathbb{T}}

\renewcommand{\P}{\mathbb{P}}
\newcommand{\bP}{\mathbb{P}}
\newcommand{\bH}{\mathbb{H}}
\newcommand{\bF}{\mathbb{F}}

\newcommand{\nosum}{}

\def\d{\partial}

\author{M\'at\'e Gerencs\'er}
\institute{TU Wien}
\begin{document}
\title{Analytically weak solutions to stochastic heat equations with spatially rough noise }
\maketitle
\begin{abstract}
In \cite{long} the authors showed existence and uniqueness of solutions to the nonlinear one-dimensional stochastic heat equation driven by a Gaussian noise that is white in time and rougher than white in space (in particular, its covariance is not a measure). Here we present a simple alternative to derive such results by considering the equations in the analytically weak sense, using either the variational approach or Krylov's $L^p$-theory. Various improvements are obtained as corollaries.
\end{abstract}

\section{Introduction}

Let $\xi$ be a $1+1$-dimensional space-time white noise and consider the nonlinear stochastic heat equation
\begin{equ}\label{eq:SHE}
\d_t u=\Delta u + \sigma(u)D^\gamma\xi,
\end{equ}
where $D$ is a differential operator of order $1$ (for most of the article it will be $D=(1-\Delta)^{\frac{1}{2}}$ with the periodic Laplacian $\Delta$).
We will take the parameter $\gamma$ to belong to $[0,\frac{1}{4})$, that is, the driving noise to be more singular in space than white.
% (naturally, $\gamma\leq 0$ would only be easier, but we consider that case well-known).
A very similar equation was considered in \cite{long}, with a slightly different choice of ``white in time and rougher than white in space'' noise (the main parameter therein, $H$, corresponds to $1/2-\gamma$ in our setup).
The authors showed existence and uniqueness of strong solutions via Walsh's martingale measure approach to SPDEs \cite{Walsh}.
The property highlighted in \cite{long} as the main difficulty  is that for $\gamma>0$ the covariance of the noise, formally written 
\begin{equ}
\E\big(D^\gamma\xi(t,x)D^\gamma\xi(s,y)\big)=\delta(t-s)D^{2\gamma}\delta(x-y)
\end{equ}
is not a positive measure, which requires quite significant modifications of the classical theory  when tackling the equation with \cite{Walsh}.
The purpose of the present article is to demonstrate that this issue is much easier to circumvent if one formulates the equation in the (analytically) weak sense using either the variational approach \cite{Pardoux,KR_SEE} or the $L^p$-theory \cite{K_Lp}.

\begin{remark}
For convenience, we consider the equation on a compact domain.
%, either with homogeneous Dirichlet or with periodic boundary conditions.
This is one difference from \cite{long}, where the equation is considered on the real line.
We leave the extensions of the weak solution approach to infinite volume for future work.
\end{remark}

\subsection{Variational approach}
Initiated by \cite{Pardoux,KR_SEE}, the variational approach is the most commonly used theory for analytically weak solutions of SPDEs.
To verify whether an equation falls into its (rather large) scope, one only needs to check a couple of simple assumptions.
Doing so for \eqref{eq:SHE}, this approach yields a very simple proof of existence of (probabilistically) weak solutions in the full regime $\gamma<\frac{1}{4}$,
(probabilistically) strong existence and uniqueness in the case of linear noise $\sigma(u)=c_1+c_2 u$ also in the full regime $\gamma<\frac{1}{4}$,
and (probabilistically) strong existence and uniqueness in the nonlinear case in a more restrictive region
%In the regime $\gamma<1/6$ (corresponding to $H>1/3$ for \cite{long}) one can get strong well-posedness of \eqref{eq:SHE} as a relatively direct application of \cite{Neelima}. For $\gamma\in[1/6,1/4)$ {\color{red}let's see...}
%Once the well-posedness is established, it is
\begin{equ}\label{eq:gamma strange bound}
%\gamma<\gamma_0:=\frac{5-\sqrt{13}}{12}= 0.1162\ldots
\gamma<\gamma_0:=\frac{5-\sqrt{21}}{4}=0.104356\ldots
\end{equ}
The linear case, fully covered by this approach, includes the example $\sigma(u)=u$, which is particularly relevant as the Cole-Hopf transform of the KPZ equation with rougher than white noise, studied also in \cite{Hoshino, Martin, Fabio}.

The restriction \eqref{eq:gamma strange bound} is of course quite unnatural and unsatisfactory and can be seen as an artifact of the $L^2$-based setting of the classical variational approach, ultimately losing regularity from Sobolev embeddings.
In order to overcome this restriction, we shall appeal to the $L^p$-theory of SPDEs \cite{K_Lp}, see Section \ref{sec:Lp} below.

To set up the result, consider \eqref{eq:SHE} on the torus $\T=\R/(2\pi\Z)$ (or equivalently, on the interval $[0,2\pi]$ with periodic boundary conditions) and on the fixed time horizon $[0,1]$.
Let $e^n(x)=\pi^{-\frac{1}{2}}\cos(nx)$ for $n> 0$, $e^n(x)=\pi^{-\frac{1}{2}}\sin(nx)$ for $n<0$, and $e^0(x)=(2\pi)^{-\frac{1}{2}}$. 
The sequence $(e^n)_{n\in\Z}$ is an orthonormal basis of $L^2(\T)$ and one has $-\Delta e^n=n^2e^n$.
The eigenvalues of $1-\Delta$ are denoted by $\scal{n}^2:=1+n^2$.
%Note that $D^\gamma e_n=\bar n^\gamma e_n$, where $\bar n=n/2$ if $n$ is even and $\bar n=(n+1)/2$ if $n$ is odd.
We denote by $\cD'(\T)$ the space of distributions on $\T$.
Define the Sobolev norms: for $f\in\cD'(\T)$, set
\begin{equ}
\|f\|_{H^\alpha}^2=\sum_{n\in\Z}\scal{n}^{2\alpha}|f(e^n)\big|^2.
\end{equ}
The space of distributions with finite $\|\cdot\|_{H^\alpha}$ norm is denoted by $H^\alpha$.
Of course $H^\alpha$ comes with a natural inner product, but it will be convenient to use more generally the notation
\begin{equ}
(f,g)_\alpha=\sum_{n\in\Z}\scal{n}^{2\alpha}f(e_n)g(e_n),
\end{equ}
which is well-defined whenever $f\in H^{\alpha+\beta}$ and $g\in H^{\alpha-\beta}$ for some $\beta\in\R$.
In the case $\alpha=\beta=0$ we will sometimes also use the notation $(f,g)_{L^2}$ to keep the notation instructive.
Powers of $1-\Delta$ are defined as the closures of the linear extensions of $e^n\mapsto (1-\Delta)^{\frac{\beta}{2}} e^n:=\scal{n}^\beta e^n$. The map $(1-\Delta)^{\frac{\beta}{2}}$ is an isometry from $H^\alpha$ to $H^{\alpha-\beta}$, for any $\alpha,\beta\in\R$.
Recall the elementary fact that for any $\alpha<\beta<\delta$, $\eps>0$, there exists a constant $C=C(\alpha,\beta,\delta,\eps)$ such that for all $x>0$ one has $x^\beta\leq\eps x^\delta+C x^\alpha$. As a consequence,
\begin{equ}\label{eq:interpolation}
\|f\|_{H^\beta}^2\leq \eps\|f\|_{H^{\delta}}^2+C\|f\|_{H^\alpha}^2.
\end{equ}
Some other common function spaces that we encounter are given by
%, for $A=\R$ or $A=\T$, by
\begin{equs}
f&\in C^\alpha(A)\quad &&\Leftrightarrow\quad \|f\|_{C^\alpha(A)}:=\sup_{x\in A}|f(x)|+\sup_{x\neq y\in A}\frac{|f(x)-f(y)|}{|x-y|^\alpha}<\infty\quad&&\alpha\in(0,1],
\\
f&\in C^\alpha(A)\quad &&\Leftrightarrow\quad\|f\|_{C^\alpha(A)}:=\sum_{k=0}^{\lceil\alpha\rceil-1}\|\d^k f\|_{C^{\alpha}(A)}<\infty\quad&&\alpha>1,
\\
g&\in L^p(A)\quad &&\Leftrightarrow\quad\|g\|_{L^p(A)}^p:=\int_{A}|g(x)|^p\,dx<\infty\quad&&p\in[1,\infty).
\end{equs}
The domain $A$ is often dropped from the notation. The space of continuous functions from $A$ to $B$ (with metric spaces $A$ and $B$) with the supremum norm is denoted by $C(A;B)$.
The notation $\subset$ between function spaces always denotes continuous embeddings.

Take a complete probability space $(\Omega,\cF,\P)$ and recall that (as far as an equation on $[0,1]\times \T$ is concerned) the white noise $\xi$ is an isometry from $L^2([0,1]\times \T)$ to $L^2(\Omega)$ such that each $\xi(\varphi)$ is Gaussian.
Define the processes $(W^n_t)_{t\in[0,1]}$ for $n\in\Z$ as the continuous modifications of the stochastic process $\big(\xi(\bone_{[0,t]}\otimes e^n)\big)_{t\in[0,1]}$. It is standard that such continuous modifications exist and $(W^n)_{n\in\Z}$ is a sequence of mutually independent Brownian motions.
Let $\bF=(\cF_t)_{t\in[0,1]}$ be a complete filtration such that each $W^n$ is an $\bF$-Brownian motion.
Let $\cP$ be the predictable $\sigma$-algebra on $\Omega\times[0,1]$ based on $\bF$.
As we will always work with analytically weak solutions, in the sequel the notions ``strong'' and ``weak'' refer to the distinction of solution concepts in the probabilistic sense. They are defined as follows. The choice of the $\mu$ is postponed, for now it may be taken as an arbitrary real parameter.

\begin{definition}\label{def:SHE}
A strong solution to \eqref{eq:SHE} with initial condition $\psi\in H^{\mu-1}$ is
% a $\cP\otimes\cB(I)$-measurable map $u:\Omega\times[0,1]\times I\to\R$ that belongs to $L^2(\Omega\times[0,1]; H^{\mu})$ and $L^2(\Omega;C([0,1];H^{\mu-1}))$
a $\bF$-adapted $H^{\mu-1}$-valued continuous stochastic process $u$ that belongs to $L^2(\Omega\times[0,1]; H^{\mu})$, such that for all $v\in H^{\mu}$ each expression in the equality
\begin{equ}\label{eq:SHE-def-weak}
(u_t,v)_{\mu-1}=(\psi,v)_{\mu-1}+\int_0^t(\Delta u_s,v)_{\mu-1}\,ds
+\sum_{n\in\Z}\int_0^t(\sigma(u_s)\scal{n}^\gamma e^n,v)_{\mu-1}\,dW^n_s
\end{equ}
is well-defined almost surely for all $t\in[0,1]$ and the equality holds.
\end{definition}
\begin{definition}\label{def:SHE-weak}
A weak solution to \eqref{eq:SHE} with initial condition $\psi\in H^{\mu-1}$ is a
collection $\{(\bar\Omega,\bar\cF,\bar\bP),\bar\bF,(\bar W^n)_{n\in\N},\bar u\}$ such that
$(\bar\Omega,\bar\cF,\bar\bP)$ is a complete probability space,
$\bar\bF$ is a complete filtration of $\bar\cF$,
$(\bar W^n)_{n\in\N}$ is a sequence of independent $\bar\bF$-Brownian motions,
$\bar u$ is a $\bar\bF$-adapted $H^{\mu-1}$-valued continuous stochastic process that belongs to $L^2(\bar \Omega\times[0,1]; H^{\mu})$, such that for all $v\in H^{\mu}$ each expression in the equality
\begin{equ}\label{eq:SHE-def-WWweak}
(\bar u_t,v)_{\mu-1}=(\psi,v)_{\mu-1}+\int_0^t(\Delta \bar u_s,v)_{\mu-1}\,ds
+\sum_{n\in\Z}\int_0^t(\sigma(\bar u_s)\scal{n}^\gamma e^n,v)_{\mu-1}\,d\bar W^n_s
\end{equ}
is well-defined $\bar \bP$-almost surely for all $t\in[0,1]$ and the equality holds.
\end{definition}
%\begin{remark}
%If $\mu$ is chosen appropriately and $\sigma$ is sufficiently regular, the prescribed regularity from $u$ already implies that each term in the above equalities
%is well-defined and has a continuous modification (in the $t$ variable) . But since this fact is not immediately obvious, we prefer to keep the requirement of well-definedness as part of the definition.
%\end{remark}

The main theorems using the variational approach read as follows.

\begin{theorem}\label{thm:main-weak}
Let $\gamma\in[0,\frac{1}{4})$ and $\mu\in(\gamma,\frac12-\gamma)$. Suppose that $\psi\in H^{\mu-1}$ and that $\sigma\in C^{\frac{\gamma}{\mu}+\eps_1}$ with some $\eps_1>0$. Then there exists a weak solution to \eqref{eq:SHE} with initial condition $\psi$.
\end{theorem}

\begin{theorem}\label{thm:main-strong}
Let either of the following hold:
\begin{enumerate}[(i)]
\item $\gamma\in[0,\frac{1}{4})$, $\mu\in(\gamma,\frac12-\gamma)$, and $\sigma(u)=c_1+c_2u$ for some constants $c_1,c_2\in\R$;
\item $\gamma\in[0,\gamma_0)$, $\mu,q$ are as in Proposition \ref{prop:exponents}, and $\sigma\in C^{1+\frac{1}{q}}$.
\end{enumerate}
Let furthermore $\psi\in H^{\mu-1}$. Then there exists a unique strong solution to \eqref{eq:SHE} with initial condition $\psi$.
\end{theorem}
\begin{remark}\label{rem:threshold}
Considering \eqref{eq:SHE} in the weak form gives yet another perspective for the threshold $\frac{1}{4}$, also appearing in the martingale measure solution theory \cite{long} and regularity structures \cite{Hoshino}: in the present context this boils down to the existence of 
$\mu$ that satisfies the inequalities $\mu<\frac{1}{2}-\gamma$ and $\mu>\gamma$ which are used for the first and second inequalities in
\begin{equ}
\sum_{n\in\N}\|v \scal{n}^\gamma e^n\|_{H^{\mu-1}}^2\lesssim \|v\|_{H^\gamma}^2\leq \eps \|v\|_{H^{\mu}}^2+ C(\gamma,\mu,\eps)\|v\|_{H^{\mu-1}}^2
\end{equ}
for any $\eps>0$, see \eqref{eq:krylov-type} below. The necessity for such a bound arises naturally to get energy estimates for the equation.
What happens beyond the $\frac{1}{4}$ threshold is the subject of the recent papers \cite{Martin, Fabio} (both formulated in terms of the KPZ equation, obtained by taking $\sigma(u)=u$ and formally taking logarithm).
\end{remark}

The variational approach offers several advantages.
\begin{enumerate}[(1)]
\item \emph{Short proof}. The probabilistically strong well-posedness of \eqref{eq:SHE} in the variational framework boils down to a back-of-the-envelope verification of a few assumptions in \cite{KR_SEE} (more precisely, the variant \cite{Neelima} is more suitable, see Remark \ref{rem:loc-mon}).

\item \emph{Other boundary conditions}. The simplicity of the proof allows one to easily extend the result to other settings. For example, consider the case of homogeneous Dirichlet boundary conditions.
This amounts to changing the orthonormal basis of $L^2$ that we work with:
for $n\in\N=\{1,2,\ldots\}$ let $\bar e^n(x)=(\pi)^{-\frac{1}{2}}\sin(2^{-1}nx)$. We can then define the Dirichlet Sobolev spaces by the norms
\begin{equ}
\|f\|_{H^\alpha_0}^2:=\sum_{n\in\N}n^{2\alpha}|f(\bar e^n)\big|^2.
\end{equ}
The analogue of Theorems \ref{thm:main-weak}-\ref{thm:main-strong} for the Dirichlet case can then be obtained by essentially the same argument. We leave the details to the interested reader.
%Let $e^0(x)\equiv(2\pi)^{-1/2}$, and for $n\in \N\setminus \{0\}$ let $e^{2n}(x)=\pi^{-1/2}\cos(nx)$ and $e^{2n-1}(x)=\pi^{-1/2}\sin(nx)$.

\item \emph{Quanitative approximation}. From the theory one can also easily extract that if one takes appropriate smooth approximations of the noise, then the approximate solutions converge strongly, and this convergence can be also made quantitative.
Such ``appropriate'' approximations include, for example, spatial mollification or spectral truncation of $\xi$.
%Unlike in \cite{long}, one can therefore avoid any complications when it comes to strong existence

For illustration purposes, we formulate the approximation results for the linear equation.
%Note that to talk about mollifications, we need to extend the noise outside of $I$, which we understand in a periodic way. %Note that any function on $I$, as well as any function on $\R$ with support contained in $(-\frac{\pi}{2},\frac{\pi}{2})$, can be canonically identified with a function on the torus $\T:=\R/(2\pi\Z)$. To ease notation, we do not explicitly denote this identification, but note that on the torus the convolution operation $\ast$ is well-defined.
Let $\rho:\R\to\R$ be a smooth, nonnegative, even function with integral $1$ and support in the unit ball and for $\eps>0$ let $\rho^\eps(x)=\eps^{-1}\rho(\eps^{-1}x)$.
%Convolution in the spatial variable will be denoted by $\ast_x$
%Let $\rho_+(x)=\rho(x-1)$, $\rho^\eps(x)=\eps^{-1}\rho(\eps^{-1}x)$, $\rho^\eps_+(x)=\eps^{-1}\rho_+(\eps^{-1}x)$, and $\xi^\eps=(\rho_+^\eps\otimes\rho^\eps)\ast\xi$, understanding with the space and time variables ordered as $(t,x)$.
\begin{corollary}\label{cor:strong-approx}
Let $\gamma\in[0,\frac{1}{4})$, $\mu\in(\gamma,\frac{1}{2}-\gamma)$, and $\kappa\in(0, \frac{1}{2}-\gamma-\mu)$. Let $\psi\in H^{-\frac{1}{2}-\gamma}$.
For $\eps\in(0,1]$ and $N\in\N$ denote by $v^{(\eps)}$ and $v^{[N]}$ the solutions of
\begin{equs}
dv^{(\eps)}_t&=\Delta v^{(\eps)}_t\,dt+\sum_{n\in\Z}v^{(\eps)}_t\scal{n}^\gamma(\rho_\eps\ast e^n)\,dW^n_t,
\\
dv^{[N]}_t&=\Delta v^{[N]}_t\,dt+\sum_{|n|\leq N}v^{[N]}_t\scal{n}^\gamma e^n\,dW^n_t,
\end{equs}
both with initial condition $\psi$. Then there exists a constant $C=C(\gamma,\mu,\kappa)$ such that for all $N\in\N$ and $\eps\in(0,1)$ one has
%Then, with the convention $v^{(0)}=v^{[\infty]}=u$, one has for any $\eps,\eps'\in[0,1]$ and $N,M\in\N\cup\{\infty\}$,
\begin{equ}
\sup_{t\in[0,1]}\E\|u_t-v^{[N]}_t\|_{H^{\mu-1}}^2+\E\int_0^1\|u_s-v^{[N]}_s\|_{H^\mu}^2\leq C N^{-\kappa},
\end{equ}
\begin{equ}
\sup_{t\in[0,1]}\E\|u_t-v^{(\eps)}_t\|_{H^{\mu-1}}^2+\E\int_0^1\|u_s-v^{(\eps)}_s\|_{H^\mu}^2\leq C \eps^{\kappa}.
\end{equ}
\end{corollary}

\item \emph{Milder conditions on the data}. Whereas \cite{long} assumes $\partial_x\sigma$ to be Lipschitz and $\sigma(0)=0$ (removing the latter condition is possible but nontrivial \cite{Hu2022}), we require somewhat less: we need $\partial_x\sigma$ only to have certain H\"older continuity, without imposing anything on $\sigma(0)$.

Concerning the initial condition, \cite{long} assumes $\psi\in B^{\gamma}_{p,2}$ with $p>\frac{6}{1-4\gamma}$ for uniqueness and somewhat more for existence (see \cite[Sec.~4.2]{long} for details).
In contrast, Theorem \ref{thm:main-strong} only assumes $\psi\in H^{\mu-1}$.
For example, a frequently studied initial condition in the context of parabolic Anderson model is the Dirac-mass $\delta$, which belongs to $H^\alpha$ for any $\alpha<-\frac{1}{2}$, and so it is included in our formulation.
\item \emph{Quasilinear and degenerate quasilinear equations}. A mild formulation relies on a good control on the semigroup associated to the leading order operator and is therefore difficult to extend to quasilinear settings.
With the variational approach this seems perfectly within scope, and in the space-time white case (i.e. $\gamma=0$) this is detailed in \cite{DGG} and \cite{BGN}.

\end{enumerate}

\begin{remark}
Given that for a large class of SPDEs their weak and mild formulations are equivalent, one might wonder if the mild solution theory of \cite{DPZ} also applies to \eqref{def:SHE}.
Note that the time integral of $D^\gamma\xi$ is a cylindrical Wiener process on $H^{-\gamma}$.
Therefore to apply \cite[Thm.~7.5]{DPZ}, one needs to verify \cite[Hyp.~7.2]{DPZ}, with $U=U_0=H^{-\gamma}$ and an appropriate Hilbert space $H$. However, \cite[Hyp.~7.2]{DPZ} (iii) can not be satisfied: even in the linear case there is no function space $H$ containing the smooth functions such that for all $h\in H$ the mapping $u\mapsto h u$ belongs to $L(U,H)$. Indeed, taking $h\equiv 1$ one sees that  $U$ has to be continuously embedded in $H$, but for generic $h,u\in U$ the product $hu$ is not well-defined.

However, the other condition in \cite[Hyp.~7.2]{DPZ} (iii) is closely related to the explanation of the threshold $\frac{1}{4}$ in Remark \ref{rem:threshold}: 
with the choice $H=H^\mu$, and sufficiently regular (say, Lipschitz) $\sigma$ the bound
\begin{equ}
\Big\|t\mapsto \|u\mapsto e^{\Delta t} (\sigma(h) u)\|_{L_{\mathrm{HS}}(U,H)}\Big\|_{L^2([0,1])}\leq C(1+\|h\|_{H})
\end{equ}
is satisfied if and only if $\mu\in(\gamma,\frac{1}{2}-\gamma)$.
The verification of this is similar to \eqref{eq:krylov-type} and we leave it to the reader.
Note however that on the left-hand side $e^{\Delta t}$ is applied  \emph{not} to elements of $H$, which is somewhat different from the usual procedure of the mild solution theory.
%either $H$ has negative regularity and then the product $hu$ is not well-defined, or $H$ has (sufficient) positive regularity and then the product $hu$ is well-defined but does not belong to $
\end{remark}

\subsection{$L^p$-theory}\label{sec:Lp}
Interestingly, the possibility of taking a noise that is rougher in space than white was already mentioned $25$ years ago as a remark in Krylov's $L^p$-theory \cite[Rem.~8.9]{K_Lp}. This remark makes the guess for the range of possible exponents to be $\gamma<\frac{1}{2}$, which is the same incorrect guess that one would get by naive scaling arguments, see e.g. \cite[Sec.~1]{Hoshino}. Nevertheless, the ``message'' of \cite[Rem.~8.9]{K_Lp} is  true: $L^p$-theory can indeed be used used to prove strong well-posedness of \eqref{eq:SHE}, and one can even remove the condition \eqref{eq:gamma strange bound}.

\begin{remark}
Strictly speaking, the theory in \cite{K_Lp} is only formulated for equations in $\R$. Since the case of the torus can only ever be easier, we take the liberty of taking all results in \cite{K_Lp} valid on $\T$ as well.
\end{remark}

%We now consider the torus $\T=\R/(2\pi\Z)$ and the orthonormal basis $(e^k)_{k\in\Z}$ given by $e^k(x)=\pi^{-\frac{1}{2}}\cos(kx)$ for $k> 0$, $e^k(x)=\pi^{-\frac{1}{2}}\sin(kx)$ for $k<0$, and $e^0(x)=(2\pi)^{-\frac{1}{2}}$. As before, one has $(-\Delta)e^k=k^2e^k$, but since $0$ is now an eigenvalue, it will be more convenient to consider fractional powers of $(1-\Delta)$.
%Consider the equation
%\begin{equ}\label{eq:SHE-torus}
%\d_t u=\Delta u + \sigma(u)(1-\Delta)^{\frac{\gamma}{2}}\xi.
%\end{equ}
%Similarly to before, we fix $(\Omega,\cF,\P)$ to be a complete probability space, the white noise $\xi$ is an isometry from $L^2([0,1]\times \T)$ to $L^2(\Omega)$ such that each $\xi(\varphi)$ is Gaussian, and the processes $(W^k_t)_{t\in[0,1]}$ for $k\in\Z$ are defined as the continuous modifications of the stochastic process $\big(\xi(\bone_{[0,t]}\otimes e^k)\big)_{t\in[0,1]}$. 
%Let $\bF=(\cF_t)_{t\in[0,1]}$ be a complete filtration such that each $W^k$ is an $\bF$-Brownian motion.

For $\alpha\in\R$ and $p\in[2,\infty)$, the completion of the space of smooth functions with respect to the norm $\|f\|_{H^{\alpha,p}}:=\|(1-\Delta)^{\frac{\alpha}{2}}\|_{L^p(\T)}$ is denoted by $H^{\alpha,p}$. Similarly, the completion of sequences of smooth functions $f=(f^k)_{k\in\Z}$ with finitely many nonzero coordinates with respect to the norm $\|f\|_{H^{\alpha,p}_{\ell^2}}:=\big\|\|(1-\Delta)^{\frac{\alpha}{2}}f\|_{\ell^2}\big\|_{L^p(\T)}$ is denoted by $H^{\alpha,p}_{\ell^2}$. In the latter case $(1-\Delta)^{\frac{\alpha}{2}}$ is understood to act component-wise.
By $\bH^{\alpha,p}$ we denote the space $L^p(\Omega\times[0,1];\cP;H^{\alpha,p})$ and similarly $\bH^{\alpha,p}_{\ell^2}$.
One has the natural analogue of \eqref{eq:interpolation} from general interpolation theory (see e.g. \cite{Triebel}): for any $\alpha<\beta<\delta$, $\eps>0$, there exists a constant $C=C(\alpha,\beta,\delta,\eps)$ such that for all $x>0$ one has
\begin{equ}\label{eq:interpolationp}
\|f\|_{H^{\beta,p}}\leq \eps\|f\|_{H^{\delta,p}}+C\|f\|_{H^{\alpha,p}}.
\end{equ}
After integrating in time and $\omega$ one gets the same inequality for the $\bH^{\alpha,p}$ spaces.

The solution space is defined as follows: we say that $u\in \cH^{\alpha,p}$ if $u$ is a $\cD'(\T)$-valued random process such that $\Delta u\in\bH^{\alpha-2,p}$, $u_0\in L^p(\Omega;\cF_0;H^{\alpha-\frac{2}{p},p})$, and there exist $f\in\bH^{\alpha-2,p}$, $g\in\bH^{\alpha-1,p}_{\ell^2}$ such that for all smooth test function $\varphi$
\begin{equ}
u_t(\varphi)=u_0(\varphi)+\int_0^t f_s(\varphi)\,ds+\sum_{n\in\Z}\int_0^tg^n_s(\varphi)\,dW^n_s
\end{equ}
holds for all $t\in[0,1]$ with probability $1$. Note that the pair $(f,g)$, if exists, is uniquely determined by $u$ (see \cite[Rem.~3.3]{K_Lp}). The norm on $\cH^{\alpha,p}$ is defined by summing up the appropriate norms of $\Delta u$, $u_0$, $f$, and $g$ (by the previous remark, this quantity is well-defined).
For any stopping time $\tau$, changing the time horizon everywhere from $[0,1]$ to $[0,\tau)$, one can analogously define the spaces $\bH^{\alpha,p}(\tau)$, $\bH^{\alpha,p}_{\ell^2}(\tau)$, $\cH^{\alpha,p}(\tau)$.

We recall two embeddings of $\cH^{\alpha,p}$ spaces. By \cite[Thm.~3.7]{K_Lp}, for any $\alpha$ and $p$ one has the embedding
\begin{equ}\label{eq:Lp-embed1}
\cH^{\alpha,p}(T)\subset L^p(\Omega;C([0,T];H^{\alpha-2,p})),
\end{equ}
with the constant of the embedding independent of $T$. 
By \cite[Thms.~7.1,7.2]{K_Lp}, for any $\alpha\in(0,\frac{1}{2})$ and $\eps\in(0,\alpha)$ there exists a $p_0$ such that for all $p>p_0$ one has the embedding
\begin{equ}\label{eq:Lp-embed2}
\cH^{\alpha,p}(T)\subset L^p(\Omega;C^{\alpha-\eps}_{\textrm{par}}([0,T]\times\T)),
\end{equ}
where $C^{\beta}_{\textrm{par}}$ denotes the space of $\beta$-H\"older continuous functions with respect to the parabolic distance. In other words, with the usual Euclidean distance these are exactly the functions that are $\frac{\beta}{2}$-H\"older continuous in time and $\beta$-H\"older continuous in space.
%By \cite[Thm.~3.7]{K_Lp} one has
%\begin{equ}\label{eq:Lp-embed}
%\|u\|_{\bH^{\alpha,p}}\leq C\|u\|_{\cH^{\alpha,p}}
%\end{equ}
%with a constant independent of $u$.
%We say that $u\in \cH^{\alpha,p}_{\loc}$ if there exists a sequence of increasing stopping times $(\tau_k)_{k\in\N}$ such that almost surely there exists $K=K(\omega)$ with $\tau_K=1$ and $u_{\cdot\wedge\tau_k}\in\cH^{\alpha,p}$ for all $k\in\N$.

A process $u\in\cH^{\alpha,p}$ is called a solution to \eqref{eq:SHE} if $f=\Delta u$ and $g=G^\gamma \sigma(u)$, where $G^\gamma$ is defined by $(G^\gamma v)^n=\scal{n}^\gamma v e^n$.

\begin{theorem}\label{thm:Lp}
Let $\gamma\in[0,\frac{1}{4})$ and $\mu\in(\gamma,\frac{1}{2}-\gamma)$. Let $\sigma\in C^{1+\frac{\gamma}{\mu}+\eps_1}$ for some $\eps_1$. Then for all sufficiently large $p$, if $\psi\in H^{\mu-\frac{2}{p},p}$, then \eqref{eq:SHE} has a unique strong solution in $\cH^{\mu,p}$. Moreover, there exists a constant $C=C(\gamma,\mu,p,\eps_1,\|\sigma\|_{C^{1+\frac{\gamma}{\mu}+\eps_1}})$, such that the solution satisfies
\begin{equ}\label{eq:apriori-Lp}
\|u\|_{\cH^{\mu,p}}\leq C\big(1+\|u_0\|_{H^{\mu-\frac{2}{p},p}}\big).
\end{equ}
\end{theorem}
%\begin{remark}
%{\color{red}rewrite}The $L^p$-theory comes with estimates on the H\"older continuity of the solution. By \cite[Thms.~7.1,7.2]{K_Lp} for any $\nu\in(0,\frac{1}{2})$, $\eps>0$ there is a $p_0$ such that for all $p>p_0$ all elements of $\cH^{\nu,p}$ are $\frac{\nu}{2}-\eps$-H\"older continuous in time and $\nu-\eps$-H\"older continuous in space.
%\end{remark}

\subsection{Wong-Zakai theorem}

As an application of some of the results obtained in the previous sections, we conclude with another approximation theorem, this time by mollifying the noise in both the spatial and temporal variables, also known as Wong-Zakai approximations.
In the context of space-time white noise driven SPDEs such results date back to \cite{BMS}, widely extended by \cite{HAIRER2015} using the theory of regularity structures.
A simple alternative in the linear case is given by \cite{Gu2019}.
%. A further corollary of the aforementioned quantified rate is that one can combine it with temporal Wong-Zakai approximations for SPDEs with smooth noise and obtain a Wong-Zakai theorem for space-time approximations of the noise. 
The proof presented here, also for the linear case, is simpler still, at the price of having to take mollifications that are far narrower in time than in space.
%Once again, the statement (and proof) is especially simple in the linear case.
%The proof presented here is somewhere between the above works: working with the linear equation the Feynman-Kac formula makes the proof quite simple
%we do require much smaller
%but the proof is very simple.
\begin{theorem}\label{thm:WZ}
Let $\gamma\in[0,\frac{1}{4})$. Then there exists $A=A(\gamma)>0$ such that the following holds. Let $\psi\in C^{\frac{1}{2}-\gamma}$.
Define $\xi^\eps=(\rho^{\eps^A}\otimes\rho^\eps)\ast\xi$.
Let $v^{(\eps)}$ solve the random PDE
\begin{equ}\label{eq:WZ}
\partial_t v^{(\eps)}=\Delta v^{(\eps)}-\tfrac{1}{2}\|\rho^\eps\|_{H^\gamma}^2 v^{(\eps)}+v^{(\eps)}D^\gamma\xi^\eps
\end{equ}
with initial condition $\psi$.
Then $v^{(\eps)}\to u$ in $L^2([0,1]\times \T)$ in probability as $\eps\to 0$.
\end{theorem}
%Note that up to leading order, $\|\rho^\eps\|_{H^\gamma}^2\sim \eps^{-1-2\gamma}$.
%Space-time Wong-Zakai theorems can also be obtained from the martingale measure approach, see  for the space-time white noise case. However, the proof is expected to get again significantly more involved in the rougher than white case.
%The theory of regularity structures \cite{H0} also provides Wong-Zakai theorems 

%(and , also with more spread out temporal smoothing, when more corrections terms would appear.
%We stress however that Corollary \ref{cor:WZ} is (...

\section{Proofs}

We start by summarising some properties of Sobolev spaces. Recall another, auxiliary scale of norms, also known as Slobodeckij norms: for $\alpha\in(0,1)$, $p\in[1,\infty)$, set
\begin{equ}
\|f\|_{W^{\alpha,p}}^p:=\int_\T |f(x)|^p\,dx+\int_\T\int_\T\frac{|f(x)-f(y)|^p}{|x-y|^{1+p\alpha}}\,dx\,dy.
\end{equ}
%Here and below we use the usual H\"older and Lebesgue spaces defined by

\begin{proposition}\label{prop:Sobolev}
Let $\alpha\in(0,\frac{1}{2})$, $\beta\in(0,1]$, $\eps>0$, $p\in[2,\infty)$. Then the following hold, with the constant $C$ depending only on $\alpha,\beta,\eps,p$.
\begin{enumerate}[(i)]
\item One has $W^{\alpha,2}=H^\alpha$ and  $W^{\alpha+\eps,p}\subset H^{\alpha,p}\subset W^{\alpha-\eps,p}$.
\item One has $\|fg\|_{H^\alpha}\leq C\|f\|_{H^{\alpha}}\|g\|_{C^{\alpha+\eps}}$ and $\|fg\|_{H^{\alpha,p}}\leq C\|f\|_{H^{\alpha+\eps,p}}\|g\|_{C^{\alpha+2\eps}}$. In particular, $C^{\alpha+\eps}\subset H^{\alpha,p}$.
% Let $\beta\in(0,1]$ and $\eps>0$. Then there exists a constant $C=C(\alpha,\beta,\eps)$ such that 
\item One has $\|f\circ g\|_{H^{\alpha,p}}\leq C\|f\|_{C^\beta(\R)}(1+\|g\|_{H^{\frac{\alpha}{\beta}+\eps,p}}^\beta)$. 
%\item There exists a constant $C=C(\alpha)$ such that $\|f\ast g\|_{H^\alpha}\leq C\|f\|_{L^1}\|g\|_{H^\alpha}$ whenever $f\in L^1(\R)$ is supported on the unit ball;
\item If $\alpha\geq \frac{1}{2}-\frac{1}{p}$, then $H^\alpha\subset L^p$.
\item If $\alpha\geq \beta-\frac{1}{p}$, then $H^{\alpha,p}\subset C^\beta$.
%\item There exists a constant $C=C(\alpha)$ such that $\|f\ast g
\end{enumerate}
\end{proposition}

\subsection{Weak existence}

We follow the notation of the variational approach as in \cite[Sec.~2]{Neelima}, which extends earlier works \cite{Pardoux, KR_SEE, LR}.
Consider the Gelfand triple
\begin{equ}
V\subset H\subset V^*\quad\equiv \quad H^{\mu}\subset H^{\mu-1}\subset H^{\mu-2}.
\end{equ}
%The auxiliary Hilbert space $U$ is taken to be $\ell^2$, note that $W=(W^n)_{n\in\N}$ is clearly a cylindrical Wiener process on $U$.
The operator $A:V\mapsto V^*$ is simply taken to be $Av=\Delta v$,
%while $B:V\mapsto L^2(U,H)\equiv \ell^2(H)$
while the operators $B^n:V\mapsto H$ are defined by $B^n v=\sigma(v)\scal{n}^\gamma e^n$.
As for numerical parameters, we take $\alpha=2$, $\beta=0$, $p_0=2$ in the notation of \cite{Neelima}, the other constants will follow from the calculations below.
%, $\theta=1$, and $f\equiv 0$. We do not specify the constants $K,L$ yet, they will be results of the calculation below.

\begin{remark}\label{rem:loc-mon}
We remark that since we take $\beta=0$, $p_0=2$, the true generality of the locally monote approach \cite{LR, Neelima} is not fully required. We do use however the flexibility to allow the constant in the monotonicity condition to depend quadratically on $\|u\|_{V}$, which is the reason why the classical version of the theory \cite{Pardoux, KR_SEE} does not fit our purposes exactly.
\end{remark}

\begin{proof}[Proof of Theorem \ref{thm:main-weak}]
We first verify conditions $A_1)$, $A_3)$, and $A_4)$ in \cite{KR_SEE} (also denoted by the same numbers in \cite{Neelima}, the results of which we will use).
%let us mention that these conditions coincide with A-1, A-3, and A-4 therein.

%Then $u$ is a solution to \eqref{eq:SHE} in the sense of Definition \ref{def:SHE} if and only if it is a solution in the sense of \cite[Def~II.2.1]{KR_SEE}. 
%By \cite[Thms~II.2.1,II.2.2]{KR_SEE}, it suffices to verify the conditions $A_1)$ through $A_5)$ therein.

For $A_1)$, one requires that for any $v,v_1,v_2\in H^\mu$, the mapping $\lambda\mapsto(v,\Delta(v_1+\lambda v_2))_{\mu-1}$ is continuous. This is immediate.

For $A_4)$, one requires the existence of a constant $K$ such that for any $v\in H^\mu$, $\|\Delta v\|_{H^{\mu-2}}\leq K\|v\|_{H^\mu}$. This is also immediate, with $K=1$.

%Now we verify the conditions $A_1)$ through $A_5)$ from \cite{KR_SEE}, where we make the choice $p=2$.For $A_1)$,

%For $A_5)$, one requires $\|\psi\|_{H^{\mu-1}}<\infty$. This holds by assumption.

For $A_3)$, one requires the existence of a constant  $K$ such that for all $v\in H^\mu$,
\begin{equ}\label{eq:coercive}
2(v,\Delta v)_{\mu-1}+\sum_{n\in\Z}\|\sigma(v) \scal{n}^\gamma e^n\|_{H^{\mu-1}}^2\leq -\|v\|_{H^\mu}^2+K(1+\|v\|_{H^{\mu-1}}^2).
\end{equ}
Since $(v,\Delta v)_{\mu-1}=-\|v\|_{H^\mu}^2+\|v\|_{H^{\mu-1}}^2$, \eqref{eq:coercive} follows if one shows (with  a possibly different constant $K$)
\begin{equ}\label{eq:key}
\sum_{n\in\Z}\|\sigma(v) \scal{n}^\gamma e^n\|_{H^{\mu-1}}^2\leq K\|v\|_{H^{\mu-\eps_2}}^\delta
\end{equ}
with some $\eps_2>0$ and $\delta\leq 2$
and uses the interpolation bound \eqref{eq:interpolation}. So we will show \eqref{eq:key}.
We will frequently use the following simple trick (motivated by \cite[Lem.~8.4]{K_Lp}, see also \cite[Ex.~7.6]{DPZ}, as well as \cite[Lem.~3.4]{DGG} for a formulation closest to the one below). For any $w\in H^\gamma$ one has
\begin{equs}
\sum_{n\in\Z}\|w \scal{n}^\gamma e^n\|_{H^{\mu-1}}^2&=\sum_{n\in\Z }\scal{n}^{2\gamma}\sum_{m\in\N}\scal{m}^{2\mu-2}(w e^n,e^m)_{L^2}^2
\\
&=\sum_{m\in\Z}\scal{m}^{2\mu-2}\sum_{n\in\Z }\scal{n}^{2\gamma}(we^m,e^n)_{L^2}^2
\\
&=\sum_{m\in\Z}\scal{m}^{2\mu-2}\|w e^m\|_{H^\gamma}^2.
\end{equs}
By the condition $\mu<1/2-\gamma$,
%there exists $\eps>0$ such that $2\mu-2+2\gamma+2\eps<-1$.
one has that $\eps:=\frac{1}{3}(1-2\mu+2\gamma)>0$ and $2\mu-2+2\gamma+2\eps<-1$.
Using Proposition \ref{prop:Sobolev} (ii), we get
\begin{equs}
\sum_{n\in\Z}\|w \scal{n}^\gamma e^n\|_{H^{\mu-1}}^2&\lesssim \|w\|_{H^\gamma}^2 \sum_{m\in\Z}\scal{m}^{2\mu-2}\|e^m\|_{C^{\gamma+\eps}}^2
\\
&\leq\|w\|_{H^\gamma}^2 \sum_{m\in\Z}\scal{m}^{2\mu-2+2\gamma+2\eps}
\\
&\lesssim\|w\|_{H^\gamma}^2.\label{eq:krylov-type}
%\\
%&\lesssim \sum_{m\in\N}m^{2\kappa}|v|_{H^\gamma}^2|e_m|_{C^\gamma}^2
\end{equs}
%where the smallness assumption on $\eps'$ was used in the last inequality.
Applying this with $w=\sigma(v)$ and using Proposition \ref{prop:Sobolev} (iii) (with $\alpha=\gamma$, $\beta=\frac{\gamma}{\mu}+\eps_1$ and $\eps$ small enough so that $\eps_2:=\mu-\frac{\gamma}{\frac{\gamma}{\mu}+\eps_1}-\eps>0$)
we obtain \eqref{eq:key}.

It immediately follows that the same bounds hold for the equations
\begin{equ}
dv^{(N)}_t=\Delta v^{(N)}_t\,dt+\sum_{|n|\leq N}(\rho_{\frac{1}{N}}\ast\sigma)(v^{(N)}_t)\scal{n}^\gamma e^n\,dW^n_t.
\end{equ}
%where in \eqref{eq:SHE-def-weak} the summation in $n$ is restricted to $n<N$.
For any $N\in\N$ this equation has a unique strong solution $v^{[N]}$, and by \cite[Thm.~2.2]{Neelima}, the above verified conditions $A_3)$ and $A_4)$ imply the uniform bound
\begin{equ}
\sup_{N\in\N}\Big(\E\sup_{t\in[0,1]}\|v^{(N)}_t\|_{H^\mu}^2+\E\int_0^1\|v^{(N)}_{t}\|_{H^{\mu-1}}^2\,dt\Big)<\infty.
\end{equ}
One can conclude the existence of weak solutions from such a uniform bound by compactness arguments, see e.g \cite[Sec.~4]{DGG}.
%We emphasise that for the verification of A-1, A-3, and A-4 we have not used the additional assumption $\gamma<1/6$ and {\color{red}Additional stuff on $\mu$}
\iffalse
Therefore, 
\begin{equ}
\sum_{n\in\N}\|\sigma(v) n^\gamma e^n\|_{H^{\mu-1}}^2\leq C(\gamma,\eps,\|\sigma\|_{C^1})(1+\|v\|_{H^\gamma}^2)\sum_{m\in\N}m^{2\mu-2+2\gamma+2\eps}.
\end{equ}
Bu the assumption on $\mu$, $2\mu-2+2\gamma<-1$, therefore for sufficiently small $\eps$ the sum converges and we obtain \eqref{eq:key}.
\fi
\end{proof}

\subsection{Strong existence and uniqueness - variational theory}
Before the proof  introduce a few auxiliary exponents.
%, some of which in fact are used even in the formulation of Theorem \ref{thm:main-strong}.
\begin{proposition}\label{prop:exponents}
Let $\gamma\geq 0$ satisfy \eqref{eq:gamma strange bound}.
Then there exists $\mu\in(\gamma,\frac12-\gamma)$, $q\in(1,\frac{\mu}{\gamma})$ such that
defining $ p=\frac{1}{1-\frac{1}{ q}}$, $\kappa=\frac{1}{2 q}$, $\theta=1+\kappa-\mu$, and $\alpha_0=\frac{2}{ q}\frac{1}{1-\theta}$,
we have $\theta\in(0,1)$ and $\alpha_0< 2$.  Moreover,
\begin{enumerate}[(a)]
\item $\Big(1+2\gamma-\frac{1+2\mu}{q}\Big)p<1$;
\item $H^\kappa\subset L^{2p}$;
\item $\|\cdot\|_{H^{\kappa}}\leq \|\cdot\|_{H^{\mu}}^\theta\|\cdot\|_{H^{\mu-1}}^{1-\theta}$.
\end{enumerate}
\end{proposition}
\begin{proof}
The inequality $\theta>0$ is trivial. Concerning the inequalities $\theta<1$ and $\alpha_0<2$, by continuity it suffices to verify them with $\mu=\frac{1}{2}-\gamma$ and $q=\frac{\mu}{\gamma}$.
With this choice, rewriting $\theta$ only in terms of $\gamma$, one has $ \theta(\gamma)=\frac{1}{2}-\gamma+\frac{\gamma}{1-2\gamma}$, from where it is easy to verify that $\theta(\gamma)<1$ on $[0,\frac{3-\sqrt{5}}{4})\supset[0,\gamma_0]$, as claimed.
Similarly, rewriting $\alpha_0$ purely in terms of $\gamma$, one has
\begin{equ}
%\tilde\alpha=\tilde\alpha(\gamma)=2\frac{\gamma}{\frac{1}{2}-\gamma}\frac{\frac{\gamma}{1-2\gamma}+\frac{1}{2}+\gamma}{\frac{1}{2}-\gamma-\frac{\gamma}{1-2\gamma}}=\frac{4\gamma(4\gamma^2-2\gamma-1)}{(2\gamma-1)(4\gamma^2-6x+1)}.
\alpha_0(\gamma)=\frac{2\gamma}{\frac{1}{2}-\gamma}\frac{1}{\frac{1}{2}-\gamma-\frac{\gamma}{1-2\gamma}}=\frac{8\gamma}{4\gamma^2-6\gamma+1}.
\end{equ}
From here it is easy to verify that $\alpha_0(\gamma_0)=2$, $\tilde\alpha_0(0)=0$, and $\alpha_0$ is strictly increasing on $[0,\frac{3-\sqrt{5}}{4})\supset[0,\gamma_0]$, implying the claim.

Claim (a) follows once we bound $-\frac{2\mu}{q}< -2\gamma$ in the parentheses on the left-hand side.
Claim (b) follows from Sobolev embedding (see Proposition \ref{prop:Sobolev} (iv)).
Claim (c) follows from $\kappa=\theta\mu+(1-\theta)(\mu-1)$, $\theta\in[0,1]$, and the definition of the Sobolev norms.
\end{proof}

\begin{proof}[Proof of Theorem \ref{thm:main-strong}]
By \cite[Thms~2.3-2.4]{Neelima}, Theorem \ref{thm:main-strong} follows once conditions A-1 through A-4 in \cite{Neelima} are satisfied. We have already checked that A-1, A-3, and A-4 hold even under the more relaxed conditions of Theorem \ref{thm:main-weak}.
It remains to verify A-2, for which one requires the existence of a constant $K$ such that for all $u,v\in H^{\mu}$,
\begin{equ}
2(u-v,\Delta u-\Delta v)_{\mu-1}+\sum_{n\in\Z}\big\|\big(\sigma(u)-\sigma(v)\big)\scal{n}^\gamma e^n\big\|_{H^{\mu-1}}^2\leq K\|u-v\|_{H^{\mu-1}}^2(1+\|u\|_{H^\mu}^2).
\end{equ}
In the case (i), the equation is linear, and so coercivity (i.e. A-3) implies monotonicity (i.e. A-2), so there is nothing more to prove. We therefore focus on the case (ii).
As before, using that $(u-v,\Delta u-\Delta v)_{\mu-1}=-\|u-v\|_{H^\mu}^2+\|u-v\|_{H^{\mu-1}}^2$, it suffices to show that
\begin{equ}\label{eq:monoton2}
\sum_{n\in\N}\big\|(\sigma(u)-\sigma(v))\scal{n}^\gamma e^n\big\|_{H^{\mu-1}}^2\leq 2\|u-v\|_{H^\mu}^2+K\|u-v\|_{H^{\mu-1}}^2(1+\|u\|_{H^\mu}^2).
\end{equ}
%%%% At this point could specify \lessssim if one really wants...
From \eqref{eq:krylov-type} we have that
\begin{equ}\label{eq:00}
\sum_{n\in\N}\big\|(\sigma(u)-\sigma(v))\scal{n}^\gamma e^n\big\|_{H^{\mu-1}}^2\lesssim \|\sigma(u)-\sigma(v)\|_{H^{\gamma}}^2.
\end{equ}
%with some $C=C(\gamma,\mu)$.
Recall the elementary inequality
\begin{equ}\label{eq:ineq-sigma}
|\sigma(a)-\sigma(b)-\sigma(c)+\sigma(d)|\leq \|\sigma\|_{C^1}|a-b-c+d|+\|\sigma\|_{C^{1+\frac{1}{q}}}|a-b||a-c|^{\frac{1}{q}}.
\end{equ}
%This implies 
%\begin{equs}\label{eq:01}
%\|&\sigma(u)-\sigma(v)\|_{H^{\gamma}}^2
%\\
%&\leq \|\sigma\|_{C^1}^2\|u-v\|_{H^\gamma}^2+\|\sigma\|_{C^{1+\frac{1}{q}}}^2\int_\T\int_\T\frac{|u(x)-v(x)||u(x)-u(y)|^{\frac{1}{q}}}{|x-y|^{1+2\gamma}}\,dx\,dy
%\end{equs}
Therefore we can continue \eqref{eq:00} as
\begin{equ}\label{eq:01}
\sum_{n\in\N}\big\|(\sigma(u)-\sigma(v))\scal{n}^\gamma e^n\big\|_{H^{\mu-1}}^2\lesssim I+J,
\end{equ}
where
\begin{equ}
I=%C \|\sigma\|_{C^1}^2
\|u-v\|_{H^\gamma}^2,\,\,\,
J=
%C\|\sigma\|_{C^{1+\frac{1}{q}}}^2
\int_\T\int_\T\frac{|u(x)-v(x)||u(x)-u(y)|^{\frac{1}{q}}}{|x-y|^{1+2\gamma}}\,dx\,dy.
\end{equ}
For $I$ we simply recall \eqref{eq:interpolation} to get
\begin{equ}\label{eq:05}
I\leq \eps\|u-v\|_{H^{\mu}}^2+C(\eps)\|u-v\|_{H^{\mu-1}}^2
\end{equ}
%Denote the integral on the right-hand side of \eqref{eq:01} by $I$.
%with some $C'=C'(\gamma,\mu,\|\sigma\|_{C^1})$.
for any $\eps>0$.
For $J$ we use the exponents from Proposition \ref{prop:exponents}. By H\"older's inequality with exponents $p$ and $q$, followed by properties (a), (b), and (c)
\begin{equs}[eq:intermediate]
J&\lesssim \Bigg(\int_\T\int_\T\frac{|u(x)-v(x)|^{2p}}{|x-y|^{\big(1+2\gamma-\frac{1+2\mu}{q}\big)p}}\,dx\,dy\Bigg)^{\frac{1}{p}}\Bigg(\int_\T\int_\T\frac{|u(x)-u(y)|^2}{|x-y|^{1+2\mu}}\,dx\,dy\Bigg)^{\frac{1}{q}}
\\
&\lesssim \|u-v\|_{L^{2p}}^2\|u\|_{H^\mu}^{\frac{2}{q}}
\\
&\lesssim \|u-v\|_{H^\kappa}^2\|u\|_{H^\mu}^{\frac{2}{q}}
\\
&\leq \|u-v\|_{H^{\mu}}^{2\theta}\|u-v\|_{H^{\mu-1}}^{2(1-\theta)}\|u\|_{H^\mu}^{\frac{2}{q}}.
\end{equs}
It remains to use Young's inequality $a^\theta b^{1-\theta}\leq \delta^{\frac{1}{\theta}}a+\delta^{\frac{-1}{1-\theta}}b$ with $\delta=\eps^{\theta}(1+\|u\|_{H^{\mu}})^{-\frac{\theta}{q}}$,
% where $C''$ is obtained from the implicit constant in \eqref{eq:intermediate},
so that we obtain
%, with another constant $C'''$,
\begin{equs}
J&\leq \Big(\eps\|u-v\|_{H^{\mu}}^{2}(1+\|u\|_{H^{\mu}})^{-\frac{2}{q}}+C(\eps)
\|u-v\|_{H^{\mu-1}}^{2}(1+\|u\|_{H^{\mu}})^{\frac{2\theta}{q(1-\theta)}}\Big)
\|u\|_{H^\mu}^{\frac{2}{q}}
\\
&\leq \eps\|u-v\|_{H^{\mu}}^2+C(\eps)\|u-v\|_{H^{\mu-1}}^2(1+\|u\|_{H^{\mu}})^{\alpha_0}.\label{eq:I2-final}
\end{equs}
Since $\alpha_0< 2$, combining \eqref{eq:01}, \eqref{eq:05}, \eqref{eq:I2-final}, and choosing $\eps>0$ small enough, we get \eqref{eq:monoton2}, concluding the proof.
\end{proof}
%\begin{remark}\label{rem:trivi}
%It is clear from the argument above that \eqref{eq:monoton2} also holds with any other positive constant in place of $2$, if one changes the constant $K$ appropriately.
%\end{remark}

\subsection{Strong spatial approximation}
\begin{proof}[Proof of Corollary \ref{cor:strong-approx}]
The proofs for the two approximations are fairly similar, we start with the somewhat easier case of the spectral cutoff.
%As already noted in the proof of Theorem \ref{thm:main-weak}, we have the uniform in $N$ bound
%\begin{equ}
%\sup_{N\in\N\cup\{\infty\}}\Big(\E\sup_{t\in[0,1]}\|v^{[N]}_t\|_{H^{\mu'}}^2+\E\int_0^1\|v^{[N]}_{t}\|_{H^{\mu'-1}}^2\,dt\Big)<\infty.
%\end{equ}
The difference $u-v^{[N]}$ satisfies the equation
\begin{equ}
d (u_t-v^{[N]}_t)=\Delta (u_t-v^{[N]}_t)\,dt+\sum_{n\in\Z}\scal{n}^\gamma\Big(\big(u_t-v^{[N]}_t\big)e^n\mathbf{1}_{|n|\leq N}+u_t e^n\mathbf{1}_{|n|>N}\Big)\,dW^n_t
\end{equ}
with initial condition $0$.
From It\^o's formula and the elementary inequality $(a+b)^2\leq 2a^2+2b^2$ we get
\begin{equs}
\|u_t-v^{[N]}_t\|_{H^{\mu-1}}^2\leq &\int_0^t2(u_s-v^{[N]}_s,\Delta u_s-\Delta v^{[N]}_s)_{\mu-1}
\\
&\qquad+2\sum_{|n|\leq N}\big\|\big(u_s-v^{[N]}_s\big)\scal{n}^\gamma e^n\big\|_{H^{\mu-1}}^2
\\
&\qquad+2\sum_{|n|>N}\big\|u_s \scal{n}^\gamma e^n\big\|_{H^{\mu-1}}^2\,ds+M_t,
\end{equs}
where $M_t$ is a martingale. Using \eqref{eq:krylov-type} and \eqref{eq:interpolation} as by now usual, we get, with some constant $K$,
\begin{equs}
\|u_t-v^{[N]}_t\|_{H^{\mu-1}}^2\leq &\int_0^t-\|u_s-v^{[N]}_s\|_{H^\mu}^2\,ds+\int_0^tK\|u_s-v^{[N]}_s\|_{H^{\mu-1}}^2\,ds
\\
&\qquad+\int_0^t2\sum_{|n|>N}\big\|u_s \scal{n}^\gamma e^n\big\|_{H^{\mu-1}}^2\,ds+M_t.
\end{equs}
Note that one trivially has $\scal{n}^\gamma\mathbf{1}_{n>N}\leq \scal{n}^{\gamma+\kappa}N^{-\kappa}$ for any $\kappa>0$. For the $\kappa$ of the theorem we even have $\mu<\frac{1}{2}-(\gamma-\kappa)$.
Using these observations  and \eqref{eq:krylov-type} once more (with $\gamma+\kappa$ in place of $\gamma$) to bound the last term, we have
\begin{equs}
\|u_t-v^{[N]}_t\|_{H^{\mu-1}}^2\leq &\int_0^t-\|u_s-v^{[N]}_s\|_{H^\mu}^2\,ds+\int_0^tK\|u_s-v^{[N]}_s\|_{H^{\mu-1}}^2\,ds
\\
&\qquad+KN^{-\kappa}\int_0^t\big\|u_s\big\|_{H^{\gamma+\kappa}}^2\,ds+M_t.\label{eq:approx-intermediate}
\end{equs}
Note that the last integral has finite expectation since $\gamma+\kappa<\frac{1}{2}-\mu<\frac{1}{2}-\gamma$ and so by Theorem \ref{thm:main-strong} with $\gamma+\kappa$ in place of $\mu$ one gets $u\in L^2(\Omega\times[0,1];H^{\gamma+\kappa})$.
Therefore by simply bounding the first integral by $0$, taking expectation, and using Gronwall's lemma we get
\begin{equ}
\sup_{t\in[0,1]}\E\|u_t-v^{[N]}_t\|_{H^{\mu-1}}^2\lesssim N^{-\kappa},
\end{equ}
and substituting this bound back into \eqref{eq:approx-intermediate} after taking expectations therein, we get
\begin{equ}
\E\int_0^1\|u_s-v^{[N]}_s\|_{H^\mu}^2\,ds\lesssim N^{-\kappa},
\end{equ}
as claimed.

Consider now the approximation by spatial mollification. Denote $e_\eps^n=\rho^\eps\ast e^n$. The difference $v-v^{(\eps)}$ satisfies
\begin{equ}
d (u_t-v^{(\eps)}_t)=\Delta (u_t-v^{(\eps)}_t)\,dt+\sum_{n\in\Z}\scal{n}^\gamma\Big(\big(u_t-v^{(\eps)}_t\big)e^n_\eps+u_t (e^n-e^n_\eps)\Big)\,dW^n_t
\end{equ}
with initial condition $0$.
From It\^o's formula and the elementary inequality $(a+b)^2\leq 2a^2+2b^2$ we get
\begin{equs}
\|u_t-v^{(\eps)}_t\|_{H^{\mu-1}}^2\leq &\int_0^t2(u_s-v^{(\eps)}_s,\Delta u_s-\Delta v^{(\eps)}_s)_{\mu-1}+2\sum_{n\in\Z}\big\|\big(u_s-v^{(\eps)}_s\big)\scal{n}^\gamma e^n_\eps\big\|_{H^{\mu-1}}^2
\\
&\qquad+2\sum_{n\in\Z}\big\|u_s \scal{n}^\gamma (e^n-e^n_\eps)\big\|_{H^{\mu-1}}^2\,ds+M_t,\label{eq:approx-molli}
\end{equs}
where $M_t$ is a martingale. To follow the same steps as before, we need some modifications to \eqref{eq:krylov-type}. 
%Recall that any function on $I$ can be identified with a function on $\T$, and convolutions with $\rho_\eps$ (or with the Dirac-delta $\delta$, corresponding to $\eps=0$) is always understood with the latter.
Let $v\in H^\gamma$ and let $\varrho$ be either $\rho_\eps$ or $\delta-\rho_\eps$.
By basic properties of the convolution and $L^2$ inner product
\begin{equs}
\sum_{n\in\Z}\|v \scal{n}^\gamma \varrho\ast e^n\|_{H^{\mu-1}}^2&=\sum_{n\in\Z }\scal{n}^{2\gamma}\sum_{m\in\N}\scal{m}^{2\mu-2}(v \varrho\ast e^n,e^m)_{L^2}^2
%\\
%&=\sum_{m\in\Z}m^{2\mu'-2}\sum_{n\in\N }n^{2\gamma}\scal{v e^m,\varrho\ast e^n}_{L^2(I)}^2
%\\
%&=\sum_{m\in\N}m^{2\mu'-2}\sum_{n\in\N }n^{2\gamma}\scal{v e^m,\varrho\ast e^n}_{L^2(\T)}^2
%\\
%&=\sum_{m\in\N}m^{2\mu'-2}\sum_{n\in\N }n^{2\gamma}\scal{\varrho\ast(v e^m), e^n}_{L^2(\T)}^2
\\
&=\sum_{m\in\Z}\scal{m}^{2\mu-2}\sum_{n\in\Z }\scal{n}^{2\gamma}(\varrho\ast(v e^m), e^n)_{L^2}^2
\\
&=\sum_{m\in\Z}\scal{m}^{2\mu-2}\|\varrho\ast(v e^m)\|_{H^\gamma}^2.
%\\
%&\lesssim \sum_{m\in\N}m^{2\mu-2}\|\sigma(v)e^m\|_{H^\gamma}^2
%\\
%&\lesssim \|\sigma(v)\|_{H^\gamma}^2.\label{eq:krylov-type2}
\end{equs}
%Recall the Slobodeckij representation of the Sobolev norm from Proposition \ref{prop:Sobolev} (i) and that this norm is also equivalent to the periodic Sobolev norm $\|f\|_{H^\alpha(\T)}=\|(1-\Delta_\T)^{\frac{\alpha}{2}}f\|_{L^2(\T)}$, where $\Delta_\T$ is the periodic Laplacian. This is again folklore, for this exact for see e.g. \cite[Prop.~4.10]{BGN}.
%For the periodic Sobolev norm it is easy to see that
Since the operator $(1-\Delta)^\frac{\gamma}{2}$ can be written as a convolution, it commutes with $\varrho\,\ast$. 
In the case, $\varrho=\rho_\eps$ by Young's convolutional inequality and the fact $\|\rho_\eps\|_{L^1}=1$ one therefore gets $\|\rho_\eps\ast w\|_{H^\gamma}\leq \|w\|_{H^\gamma}$.
Combining this with \eqref{eq:krylov-type} we get
\begin{equ}
\sum_{n\in\Z}\|v \scal{n}^\gamma \rho_\eps\ast e^n\|_{H^{\mu-1}}^2\lesssim \|v\|_{H^\gamma}^2.
\end{equ}
The case $\varrho=\delta-\rho_\eps$ is slightly less obvious, but let us recall (see e.g. \cite[Prop.~4.10]{BGN}) that $\|(\delta-\rho_\eps)\ast w\|_{L^2}\lesssim \eps^\kappa\|w\|_{H^\kappa}$ for $\kappa\in(0,1)$. Smuggling in $(1-\Delta)^\frac{\gamma}{2}$ again, we get $\|(\delta-\rho_\eps)\ast w\|_{H^\gamma}\lesssim \eps^\kappa\|w\|_{H^{\gamma+\kappa}}$. So using \eqref{eq:krylov-type} once more (with $\gamma+\kappa$ in place of $\gamma$) we get 
\begin{equ}
\sum_{n\in\Z}\|v \scal{n}^\gamma (\delta-\rho_\eps)\ast e^n\|_{H^{\mu-1}}^2\lesssim \eps^\kappa\|v\|_{H^{\gamma+\kappa}}^2.
\end{equ}
Using these two bounds in \eqref{eq:approx-molli} we get
\begin{equs}
\|u_t-v^{(\eps)}_t\|_{H^{\mu-1}}^2\leq &\int_0^t-\|u_s-v^{(\eps)}_s\|_{H^\mu}^2\,ds+\int_0^tK\|u_s-v^{(\eps)}_s\|_{H^{\mu-1}}^2\,ds
\\
&\qquad+K\eps^{\kappa}\int_0^t\big\|u_s\big\|_{H^{\gamma+\kappa}}^2\,ds+M_t,
\end{equs}
and the proof is concluded by the same argument that followed \eqref{eq:approx-intermediate}.
\end{proof}

\subsection{Strong existence and uniqueness - $L^p$-theory}
We need an $L^p$ version of \eqref{eq:krylov-type}, extending \cite[Lem.~8.4]{K_Lp}.
Recall that $G^\gamma$ is defined by $(G^\gamma v)^n=\scal{n}^\gamma v e^n$.
\begin{lemma}\label{lem:Krylov-lemma}
Let $\gamma\in[0,2]$, $\mu<\frac{1}{2}-\gamma$, and $p\in[2,\infty)$. Then there exists a constant $C=C(\gamma,\mu)$ such that for any $u\in H^{\gamma,p}$ one has
\begin{equ}\label{eq:Krylov-Lpmain}
\big\|G^\gamma u\big\|_{H^{\mu-1,p}_{\ell^2}}\leq C \|u\|_{H^{\gamma,p}}.
\end{equ}
\end{lemma}
\begin{proof}
Instead of considering $G^\gamma$ to depend on $\gamma$, let us consider the single map
%\begin{equ}
$G:\,u\mapsto (u e^k)_{k\in\Z}$
%\end{equ}
and encode the $\gamma$-dependence in the norms: for $\alpha\in\R$, $\beta_1\in\N$, $\beta_2\in\R$, and a sequence of functions $f=(f^k)_{k\in\Z}$ define
\begin{equ}
\vn{f}_{\alpha,\beta_1,\beta_2}=\Bigg(\int_{\T}\Big(\sum_{k\in\Z}(1+k^2)^{\frac{\alpha}{2}}\big|\partial^{\beta_1}(1-\Delta)^{\frac{\beta_2}{2}}f^k(x)\big|^2\Big)^{\frac{p}{2}}\,dx\Bigg)^{\frac{1}{p}}.
\end{equ}
Note that the claimed inequality \eqref{eq:Krylov-Lpmain} can be rewritten as
\begin{equ}\label{eq:Krylov-Lpmain2}
\vn{G u}_{\gamma,0,\mu-1}\lesssim \|u\|_{H^{\gamma,p}}.
\end{equ}
From \cite[Lem.~8.4]{K_Lp} one has for any $\beta_1,\beta_2$ such that $\beta_1+\beta_2<\frac{1}{2}$,
\begin{equ}\label{eq:Krylov-0}
\vn{G u}_{0,\beta_1,\beta_2-1}\lesssim \|u\|_{H^{0,p}}.
\end{equ}
For the reader's convenience let us give the proof of \eqref{eq:Krylov-0}. Recall that the operator $\partial^{\beta_1}(1-\Delta)^{\frac{\beta_2-1}{2}}$ can be written as a convolution with a function $R$ that is bounded by a constant times $(|x|^{-\beta_1-\beta_2}\vee 1)|\log (x\wedge\frac{1}{2})|$. In particular, under the assumption $\beta_1+\beta_2<\frac{1}{2}$, $R$ is square integrable on $\T$. Therefore by Parseval's identity and Young's convolutional inequality we have
\begin{equs}
\vn{G u}_{0,\beta_1,\beta_2-1}&=\Bigg(\int_{\T}\Big(\sum_{k\in\Z}\Big(\int_{\T}R(x-y)u(y)e^k(y)\,dy\Big)^2\Big)^{\frac{p}{2}}\,dx\Bigg)^{\frac{1}{p}}
\\
&=\Bigg(\int_{\T}\Big(\int_{\T}\big(R(x-y)u(y)\big)^2\,dy\Big)^{\frac{p}{2}}\,dx\Bigg)^{\frac{1}{p}}
\\
&\leq \big(\|R^2\|_{L^1}\|u^2\|_{L^\frac{p}{2}}\big)^{\frac{1}{2}}.\label{eq:222}
\end{equs}
This proves \eqref{eq:Krylov-0}.
Next note that
\begin{equs}
(1+k^2)ve^k=v(1-\Delta)e^k&=(1-\Delta)(v e^k)-\big((1-\Delta)v\big)e^k-2\partial v\partial e^k
\\
&=(1-\Delta)(v e^k)-\big((1+\Delta)v\big)e^k-2\partial (\partial v e^k).
\end{equs}
Therefore from the triangle inequality we get
\begin{equ}
\vn{G u}_{2,0,\mu-1}\leq \vn{G u}_{0,0,\mu+1}+\vn{G (1+\Delta) u}_{0,0,\mu-1}+2\vn{G\partial v}_{0,1,\mu-1}.
\end{equ}
Therefore applying \eqref{eq:Krylov-0}, we get that as long as $\mu+2<1/2$,
\begin{equ}\label{eq:Krylov-2}
\vn{G u}_{2,0,\mu-1}\lesssim \|u\|_{L^p}+\|\partial u\|_{L^p}+\|\Delta u\|_{L^p}\lesssim \|u\|_{H^{2,p}}.
\end{equ}
It follows from \cite[Sec.~1.15]{Triebel} that both of the norms $\vn{\cdot}_{\alpha,0,\beta}$ and $\|\cdot\|_{H^{\gamma,p}}$ interpolate in the natural way with respect to the parameters $\alpha,\beta,\gamma$. Therefore interpolating between \eqref{eq:Krylov-0} and \eqref{eq:Krylov-2} we get that for any $\gamma\in(0,2)$ and any $\mu$ that can be written as $\mu=\frac{\gamma}{2}\mu_1+(1-\frac{\gamma}{2})\mu_0$ with $\mu_1<-\frac{3}{2}$, $\mu_0<\frac{1}{2}$, the desired bound \eqref{eq:Krylov-Lpmain2} holds.
It remains to note that any $\mu<\frac{1}{2}-\gamma$ has this property, with choosing $\mu_0=\mu+\gamma$ and $\mu_1=\mu+\gamma-2$.
\end{proof}

\iffalse
Let us again recall some properties of Sobolev spaces. Define the
Slobodeckij spaces $W^{\alpha,p}$ defined, for $\alpha\in(0,1)$ and $p\in[1,\infty)$ by the norm
\begin{equ}
\|v\|_{W^{\alpha,p}}^p=\int_\T |v(x)|^p\,dx+\int_\T\int_\T\frac{|v(x)-v(y)|^p}{|x-y|^{1+p\alpha}}\,dx\,dy.
\end{equ} 
{\color{blue}
\begin{proposition}\label{prop:Sobolev2}
%For the others we recall some further relevant properties of Sobolev spaces $H^{\alpha,p}$, the H\"older spaces $C^\alpha$, and the 
Let $\alpha>\eps>0$.
\begin{enumerate}[(i)]
\item One has the continuous embeddings $W^{\alpha+\eps,p}\subset H^{\alpha,p}\subset W^{\alpha-\eps,p}$;
\item One has the continuous embedding $H^{\alpha+\frac{1}{p},p}\subset C^\alpha$;
\item If $\sigma\in C^1$, then 
\item If $\sigma\in C^{1+s}$
\end{enumerate}
\end{proposition}
}
\fi

\begin{proof}[Proof of Theorem \ref{thm:Lp}]
Fix some auxiliary exponents $\gamma',\mu',\nu$ such that $\gamma<\gamma'<\nu<\mu'<\mu$ and $\frac{\gamma'}{\frac{\gamma}{\mu}+\eps_1}<\nu$.
We then assume that $p$ is large enough so that $H^{\mu',p}\subset C^{\nu}$, which is possible by Proposition \ref{prop:Sobolev} (v), and that $\cH^{\mu,p}\subset L^p(\Omega;C([0,1];C^\nu))$, which is possible by \eqref{eq:Lp-embed2}.

Similarly to the variational approach, Lipschitz continuity of the diffusion coefficient will only hold locally, so we introduce a cutoff function: for all $M\geq 1$ let $\theta_M:\R\to [0,1]$ be a function that is constant $1$ on $\{|x|<M\}$, constant $0$ on $\{|x|>M+2\}$ and has Lipschitz constant at most $1$. Throughout the proof any dependendce of constants on $M$ will be explicitly stated, in particular the implicit constant in $\lesssim$ will not be allowed to depend on $M$.
Define $\Theta_M: C^{\nu}(\T)\to[0,1]$ by $\Theta_M(v)=\theta_M(\|v\|_{C^\nu})$.
 
We verify the conditions of \cite[Thm.~5.1]{K_Lp}, named Assumptions 5.1-5.6.
To reconcile the notation therein and herein, take $n+2=\mu$, $a=1$, $\sigma=0$, $f=0$, and $g(v)=\Theta_M(v)G^\gamma\sigma(v)$ (beware the clash of notation between the $f$ and $g$ here and in the setup of Section \ref{sec:Lp}, this clash also follows \cite{K_Lp}). Therefore Assumptions 5.1, 5.2., and 5.3 are trivially satisfied.

For Assumptions 5.4 and 5.5, one requires that for $v\in H^{\mu,p}$, $\Theta_M(v)G^\gamma \sigma(v)\in H^{\mu-1,p}_{\ell^2}$.
For this we note that $\Theta_M(v)$ is well-defined since $H^{\mu,p}\subset C^\nu$.
%if $p$ is large enough, using $\mu>\nu$ and Proposition \ref{prop:Sobolev2} (ii).
Then one bounds $\Theta_M(v)$ by $1$, use Lemma \ref{lem:Krylov-lemma}, and Proposition \ref{prop:Sobolev} (iii) (with $\beta=1$) to get 
\begin{equs}
\|\Theta_M(v)G^\gamma \sigma(v)\|_{H^{\mu-1,p}_{\ell^2}}&\leq\|G^\gamma \sigma(v)\|_{H^{\mu-1,p}_{\ell^2}}\lesssim\|\sigma(v)\|_{H^{\gamma,p}}\lesssim 1+\|v\|_{H^{\gamma',p}}.\label{eq:Lp-bound}
%\\
%&\lesssim \|u_0\|_{H^{\mu-\frac{2}{p},p}}+\|\sigma(u^{(M)})\|_{\bH^{\gamma,p}(T)}
%\\
%&\lesssim\|u_0\|_{H^{\mu-\frac{2}{p},p}}+ 1+\|u^{(M)}\|_{\bH^{\gamma,p}(T)}.
\end{equs}
Since $\gamma'<\mu$, this yields the required bound.

For Assumption 5.6, one requires that for all $\eps>0$ there exists $K(\eps,M)$ such that for any $v,w\in H^{\mu,p}$ one has
\begin{equ}
\Big\|\Theta_M(v)G^\gamma \sigma(v)-\Theta_M(w)G^\gamma \sigma(w)\Big\|_{H^{\mu-1,p}_{\ell^2}}\leq \eps \|v-w\|_{H^{\mu,p}}+K(\eps,M) \|v-w\|_{H^{\mu-2,p}}.
\end{equ}
By \eqref{eq:interpolationp} it suffices to show
\begin{equ}\label{eq:LpLipschitz0}
\Big\|\Theta_M(v)G^\gamma \sigma(v)-\Theta_M(w)G^\gamma \sigma(w)\Big\|_{H^{\mu-1,p}_{\ell^2}}\leq K(M) \|v-w\|_{H^{\mu',p}}
\end{equ}
with some $K(M)$. 
Without loss of generality, we may assume $\|v\|_{C^\nu}\leq \|w\|_{C^\nu}$.
%If $\|v\|_{C^\nu}\geq M+2$, then the left-hand side is $0$ so there is nothing to prove. If $\|v\|_{C^\nu}\leq M+2\leq \|w\|_{C^\nu}$, then the left-hand side equals
First we write
\begin{equ}
\Big\|\Theta_M(v)G^\gamma \sigma(v)-\Theta_M(w)G^\gamma \sigma(v)\Big\|_{H^{\mu-1,p}_{\ell^2}}= 
|\Theta_M(v)-\Theta_M(w)|\|G^\gamma \sigma(v)\|_{H^{\mu-1,p}_{\ell^2}}\Theta_{M+2}(v).
\end{equ}
Since both $\theta_M$ and the norm function has Lipschitz norm $1$, using Lemma \ref{lem:Krylov-lemma} we can write
\begin{equs}
\Big\|\Theta_M(v)G^\gamma \sigma(v)-\Theta_M(w)G^\gamma \sigma(v)\Big\|_{H^{\mu-1,p}_{\ell^2}}
&\lesssim \|v-w\|_{C^{\nu}}(1+\|v\|_{H^{\gamma,p}})\Theta_{M+2}(v)
\\
&\lesssim M\|v-w\|_{H^{\mu',p}},\label{eq:LpLipschitz1}
\end{equs}
using Proposition \ref{prop:Sobolev} (ii) and (v) in the second inequality.
%, with $\mu'>\nu$ arbitrary and $p$ large enough.
Next, recall that the inequality \eqref{eq:ineq-sigma} in particular implies that since $\sigma\in C^{1+\frac{\gamma}{\mu}+\eps_1}$, one has
\begin{equ}
\|\sigma(v)-\sigma(w)\|_{C^{\gamma'}}\lesssim \|v-w\|_{C^{\gamma'}}+\|v-w\|_{L^\infty}\|w\|^{\frac{\gamma}{\mu}+\eps_1}_{C^{\frac{\gamma'}{\frac{\gamma}{\mu}+\eps_1}}}.
\end{equ}
%Next we write, by Lemma \ref{lem:Krylov-lemma}, Proposition XXX (xx) and (Xxx)
By assumption the last unappealing exponent is bounded by $\nu$.
Therefore by Lemma \ref{lem:Krylov-lemma}, Proposition \ref{prop:Sobolev} (ii) and (v)
\begin{equs}
\Big\|\Theta_M(w)G^\gamma \sigma(v)-\Theta_M(w)G^\gamma \sigma(w)\Big\|_{H^{\mu-1,p}_{\ell^2}}&\lesssim\Theta_M(w)\|\sigma(v)-\sigma(w)\|_{H^{\gamma,p}}
\\
&\lesssim\Theta_M(w)\|\sigma(v)-\sigma(w)\|_{C^{\gamma'}}
\\
&\lesssim\Theta_M(w)\|v-w\|_{C^{\gamma'}}(1+\|w\|_{C^\nu})
\\
&\lesssim M\|v-w\|_{H^{\mu',p}}.\label{eq:LpLipschitz2}
\end{equs}
The bounds \eqref{eq:LpLipschitz1} and \eqref{eq:LpLipschitz2} imply \eqref{eq:LpLipschitz0} and so  all the assumptions of \cite[Thm.~5.1]{K_Lp} are verified.
By \cite[Thm.~5.1]{K_Lp}, for all $M\geq 1$ we get a solution $u^{(M)}\in\cH^{\mu,p}$ to the equation
\begin{equ}\label{eq:Lp-M}
du^{(M)}=\Delta u^{(M)}\,dt+\Theta_M(u^{(M)})\sum_{n\in\Z}\sigma(u^{(M)})\scal{n}^\gamma e^n\,dW^n_t,
\end{equ}
with initial condition $\psi$. Moreover, the solution is unique up to any stopping time. It remains to remove the cutoff function $\Theta_M$.

By \cite[Thm.~2.1,~Lem.~6.6]{K_Lp} one has the a priori estimate, for any $T\in(0,1]$,
\begin{equs}
\|u^{(M)}\|_{\cH^{\mu,p}(T)}&\lesssim\|\psi\|_{H^{\mu-\frac{2}{p},p}}+\|\Theta_M(u^{(M)}) G^\gamma\sigma(u^{(M)})\|_{\bH^{\mu-1,p}_{\ell^2}(T)}
\\
&\lesssim \|\psi\|_{H^{\mu-\frac{2}{p},p}}+ 1+\|u^{(M)}\|_{\bH^{\gamma',p}(T)},
\end{equs}
using \eqref{eq:Lp-bound} again in the second inequality.
\iffalse
Again we simply bound $\rho_M$ by $1$, use Lemma \ref{lem:Krylov-lemma} and Proposition \ref{prop:Sobolev2} (ii)
\begin{equs}
\|u^{(M)}\|_{\cH^{\mu,p}(T)}&\lesssim\|\psi\|_{H^{\mu-\frac{2}{p},p}}+\|G^\gamma\sigma(u^{(M)})\|_{\bH^{\mu-1,p}(T)}
\\
&\lesssim \|\psi\|_{H^{\mu-\frac{2}{p},p}}+\|\sigma(u^{(M)})\|_{\bH^{\gamma,p}(T)}
\\
&\lesssim\|\psi\|_{H^{\mu-\frac{2}{p},p}}+ 1+\|u^{(M)}\|_{\bH^{\gamma,p}(T)}.
\end{equs}
\fi
Let us use \eqref{eq:interpolationp} with $\beta=\gamma'$, $\delta=\mu-2$, $\alpha=\mu$.
% By assumption $\gamma<\mu$, so \eqref{eq:interpolationp} applies.
Moreover we can replace $\|u\|_{\bH^{\mu,p}}$ by $\|\Delta u\|_{\bH^{\mu-2,p}}$ and absorb the remainder in $\|u\|_{\bH^{\mu-2,p}}$.
Recalling that $\|\Delta u\|_{\bH^{\mu-2,p}(T)}$ is part of the $\cH^{\mu,p}$ norm, we can choose $\eps$ small enough so that this term is absorbed on the left-hand side. We end up with
\begin{equs}
\|u^{(M)}\|_{\cH^{\mu,p}(T)}^p&\lesssim \|\psi\|_{H^{\mu-\frac{2}{p},p}}^p+1+\|u^{(M)}\|_{\bH^{\mu-2,p}(T)}^p
\\
&=\|\psi\|_{H^{\mu-\frac{2}{p},p}}^p+1+\E\int_0^T\|u^{(M)}_t\|_{H^{\mu-2,p}}^p\,dt.\label{eq:Gron1}
\end{equs}
%From \cite[Thm.~3.7]{K_Lp} one has the embedding $\cH^{\mu,p}(T)\subset L^p(\Omega;C([0,T];H^{\mu-2,p}))$,
%implying in particular the estimate
%\begin{equ}\label{eq:Gron2}
By \eqref{eq:Lp-embed1} one has $\E \|u_T\|_{H^{\mu-2,p}}^p\lesssim \|u^{(M)}\|_{\cH^{\mu,p}(T)}^p$, and so one can apply Gronwall's lemma to the function $T\mapsto\E\|u^{(M)}_T\|_{H^{\mu-2,p}}^p$. Substituting back the resulting estimate in \eqref{eq:Gron1} yields in the uniform in $M$ bound
\begin{equ}\label{eq:apriori-M}
\|u^{(M)}\|_{\cH^{\mu,p}}\lesssim 1+\|\psi\|_{H^{\mu-\frac{2}{p},p}}.
\end{equ}
%By \eqref{eq:Lp-embed2} and $\nu<\mu$, we also get for large enough $p$,{\color{red}add to conditionson $p$}
Thanks to the choice of the exponents we get
\begin{equ}\label{eq:apriori-Holder-M}
\E\|u^{(M)}\|_{C([0,1];C^\nu(\T))}^p\lesssim 1+\|\psi\|_{H^{\mu-\frac{2}{p},p}}^p.
\end{equ}
Define $\tau_{N,M}=\inf\{t:\|u^{(N)}_t\|_{C^\nu}\geq M\}\wedge1$. Note that for all $N\geq M$, $u^{(N)}$ solves \eqref{eq:Lp-M} up to $\tau_{N,M}$, and therefore has to coincide with $u^{(M)}$ there. Thus, $\tau_{N,M}=\tau_{M,M}$ for all $N\geq M$ and on each stochastic interval $[0,\tau_{M,M}]$, the sequence $(u^{(N)})_{N\in\N}$ is eventually constant and the limit process solves \eqref{eq:SHE} on $[0,\tau_{M,M}]$. Since from \eqref{eq:apriori-Holder-M} it follows that $\mathbb{P}(\tau_{M,M}<1)\to 0$, we get that there exists an a.s. limit $u$ in $C([0,1];C^\nu(\T))$. Passing to $M\to\infty$ in \eqref{eq:apriori-M} yields that $u$ satisfies \eqref{eq:apriori-Lp} and in particular $u\in\cH^{\mu,p}$, and finally $u$ is a solution on $\cup_M[0,\tau_{M,M}]=[0,1]$.
\end{proof}

\begin{remark}\label{rem:apriori-Lp-eps}
Morally it is clear that the same bound \eqref{eq:apriori-Lp} also holds for $v^{(\eps)}$ defined in Corollary \ref{cor:strong-approx}, provided $\psi\in H^{\mu-\frac{2}{p},p}$. To see this more precisely, the only property used from the diffusion coefficient is \eqref{eq:Krylov-Lpmain}, so one has to verify it for $G^{\gamma,(\eps)}$ defined by
$(G^{\gamma,(\eps)} u)^n=\scal{n}^\gamma u e^n_\eps$.
Inspecting the proof of Lemma \ref{lem:Krylov-lemma}, one needs to verify \eqref{eq:Krylov-0} with $G^{(\eps)}:\,u\mapsto (u e^k_\eps)_{k\in\Z}$ in place of $G$. With $R$ as therein (of which we only need that it is square integrable) we have
\begin{equs}
\vn{G^{(\eps)} u}_{0,\beta_1,\beta_2-1}&=\Bigg(\int_{\T}\Big(\sum_{k\in\Z}\Big(\int_{\T}R(x-y)u(y)e^k_\eps(y)\,dy\Big)^2\Big)^{\frac{p}{2}}\,dx\Bigg)^{\frac{1}{p}}
\\
&=\Bigg(\int_{\T}\Big(\sum_{k\in\Z}\Big(\int_{\T}\int_{\T}\rho_\eps(y-z)R(x-z)u(z)\,dz e^k(y)\,dy\Big)^2\Big)^{\frac{p}{2}}\,dx\Bigg)^{\frac{1}{p}}
\\
&=\Bigg(\int_{\T}\Big(\int_{\T}\Big(\int_\T\rho_\eps(y-z)R(x-z)u(z)\,dz\Big)^2\,dy\Big)^{\frac{p}{2}}\,dx\Bigg)^{\frac{1}{p}}.
%\\
%&\leq \big(\|R^2\|_{L^1}\|u^2\|_{L^\frac{p}{2}}\big)^{\frac{1}{2}}.
\end{equs}
Applying first Young's convolutional inequality to the middle integral and using $\|\rho_\eps\|_{L^1}=1$, we can finish the proof of \eqref{eq:Krylov-0} as in \eqref{eq:222}, the rest of the argument then remains unchanged.
\end{remark}

\subsection{Wong-Zakai approximation}

\begin{proof}[Proof of Theorem \ref{thm:WZ}]
To ease notation, throughout the proof the summation over $n\in\Z$ is implicitly implied.
To separate the spatial and temporal scales, let $\xi^{\delta,\eps}=(\rho^\delta\otimes\rho^\eps)\ast\xi$ and denote by $u^{(\delta,\eps)}$ the solution to the random PDE
\begin{equ}
\partial_t v^{(\delta,\eps)}=\Delta v^{(\delta,\eps)}-\tfrac{1}{2}\|\rho^\eps\|_{H^\gamma}^2v^{(\delta,\eps)}+v^{(\delta,\eps)} D^\gamma\xi^{\delta,\eps}.
\end{equ}
Note that $D^\gamma\xi^{\delta,\eps}(t,x)=\nosum\scal{n}^\gamma e^n_\eps(x)\partial W^{\delta,n}_t$, where $W^{\delta,n}=\rho^\delta\ast W^n$.
Since Corollary \ref{cor:strong-approx} provides a bound for $u-v^{(\eps)}$, the main task is to bound the difference between $v^{(\eps)}$ and $v^{(\delta,\eps)}$. Since they are both spatially smooth processes for $t>0$ and continuous at $t=0$, the Feynman-Kac formula applies.
That is, taking $\hat W$ be $\sqrt{2}$ times a Brownian motion on $\T$ independent from $(W^n)_{n\in\Z}$ and denoting by $\hat\E$ the expectation with respect to $\hat W$, for any $(t,x)\in[0,1]\times\T$ we have
\begin{equ}
v^{(\delta,\eps)}_t(x)=\hat\E\bigg(\psi(x+\hat W_{t})\exp\Big(\nosum\int_0^t\scal{n}^\gamma \d W_s^{\delta,n} e^{n}_\eps(x+\hat W_{t-s})\,ds-\tfrac{1}{2}\|\rho^\eps\|_{H^\gamma}^2 t\Big)\bigg)
\end{equ}
as well as
\begin{equ}
v^{(\eps)}_t(x)=\hat\E\bigg(\psi(x+\hat W_{t})\exp\Big(\nosum\int_0^t \scal{n}^\gamma e^{n}_\eps(x+\hat W_{t-s})\,dW^n_s-\tfrac{1}{2}\int_0^t\scal{n}^{2\gamma} |e^{n}_\eps(x+\hat W_{t-s})|^2\,ds \Big)\bigg).
\end{equ}
First of all, notice that the two subtracted quantities are the same:
\begin{equs}
\nosum\scal{n}^{2\gamma}|e^{n}_\eps(x+\hat W_{t-s})|^2&=\nosum\scal{n}^{2\gamma}(\rho^\eps(x+\hat W_{t-s}-\cdot),e^n)_{L^2}^2\\
&=\|\rho^\eps(x+\hat W_{t-s}-\cdot)\|_{H^\gamma}^2=\|\rho^\eps\|_{H^\gamma}^2.
%\\&=\|\rho_\eps\|_{L^2}^2=\eps^{-1}\|\rho\|_{L^2}^2=2C^\eps.
\end{equs}
%Therefore, we have
%\begin{equ}
%\bar u^{\eps}_t(x)=\hat\E\exp\Big(\sum_{k\in \N}\int_0^t e^{\eps,k}(x+\hat W_{t-s})\,dW^k_s-C^\eps t \Big).
%\end{equ}
We can therefore write 

\begin{equs}%[eq:something in the middle0]
\,&\E| v^{(\eps)}_t(x)-v^{(\delta,\eps)}_t(x)|
\\&=\E\Big| v^{(\eps)}_t(x)
\\
&\,\times\hat \E\Big(1-\exp\Big(\nosum\int_0^t  \scal{n}^\gamma e^{n}_\eps(x+\hat W_{t-s})\d W_s^{\delta,n}\,ds-\int_0^t  \scal{n}^\gamma e^{n}_\eps(x+\hat W_{t-s})\,dW^n_s\Big)\Big)\Big|
\\
&\leq \big(\E\| v^{(\eps)}\|^2_{C([0,1]\times\T)}\big)^{\frac{1}{2}}
\\ &\,\times\Big(\tilde \E\Big|1-\exp\Big(\nosum\int_0^t   \scal{n}^\gamma e^{n}_\eps(x+\hat W_{t-s})\d W_s^{\delta,n}\,ds-\int_0^t  \scal{n}^\gamma e^{n}_\eps(x+\hat W_{t-s})\,dW^n_s\Big)\Big|^2\Big)^{\frac{1}{2}}.
\end{equs}
Here and below $\tilde \E$ denotes $\E\hat\E$.
%\begin{remark}\label{rem:intermediate}
%To bound the first term, we have two choices: either we use the embedding $H^1\subset C$ and solve the equation for $\bar u^\eps$ in the Gelfand triple $H^2\subset H^1\subset H^{0}$. The downside of this is that the noise will have to be estimated in $H^1$, which will result in a negative power of $\eps$.
%A more optimal choice is to recall Theorem \ref{thm:continuity}, or rather a slightly modified version of it (allowing to put the mollified basis instead of the real one at the noise), which gives continuity of the solution right away.
%\end{remark}
By Remark \ref{rem:apriori-Lp-eps} we have
%\begin{equ}
$\E\| u^{(\eps)}\|^2_{C_t C_x}\lesssim 1$
%\end{equ}
and therefore
\begin{equs}%[eq:something in the middle0]
\,&\E| v^{(\eps)}_t(x)-v^{(\delta,\eps)}_t(x)|
\\&\,\lesssim\Big(\tilde \E\Big|1-\exp\Big(\nosum\int_0^t  \scal{n}^\gamma e^{n}_\eps(x+\hat W_{t-s})\d W_s^{\delta,n}\,ds-\int_0^t \scal{n}^\gamma e^{n}_\eps(x+\hat W_{t-s})\,dW^n_s\Big)\Big|^2\Big)^{\frac{1}{2}}.
\end{equs}
As a function of $s$, $e^n_{\eps}(x+\hat W_{t-s})$ is a (backward) diffusion process that is independent of $(W^n)_{n\in\Z}$. Therefore we can integrate by parts 
\begin{equs}
\int_0^t  & \scal{n}^\gamma e_{\eps}^n(x+\hat W_{t-s})\d W_s^{\delta,n}\,ds-\int_0^t \scal{n}^\gamma e_{\eps}^n(x+\hat W_{t-s})\,dW^n_s
\\
&=\int_0^t  \scal{n}^\gamma (W^n_s-W^{\delta,n}_s)\,de_{\eps}^n(x+\hat W_{t-s})
\\
&\quad- \scal{n}^\gamma(W^n_t-W^{\delta,n}_t)e^n_{\eps}(x)+ \scal{n}^\gamma(W^n_0-W^{\delta,n}_0)e^{n}_\eps(x+\hat W_t)
\\
&=\int_0^t \scal{n}^\gamma (W^n_s-W^{\delta,n}_s)\d e^{n}_\eps(x+\hat W_{t-s}) \,d\hat W_{t-s}
\\
&\quad+\Bigg(\frac{1}{2}\int_0^t \scal{n}^\gamma (W^n_s-W^{\delta,n}_s)\d^2 e^{n}_{\eps}(x+\hat W_{t-s}) \,ds
\\&\qquad- \scal{n}^\gamma(W^n_t-W^{\delta,n}_t)e^{n}_{\eps}(x)+ \scal{n}^\gamma(W^n_0-W^{\delta,n}_0)e^{n}_{\eps}(x+\hat W_t)\Bigg)
\\&=:I_1+I_2.
\end{equs}
Thus,
\begin{equs}
\big(\E| v^{(\eps)}_t(x)-v^{(\delta,\eps)}_t(x)|\big)^2&\lesssim \tilde \E\big|1-\exp(I_1+I_2)\big|^2
\\
& =2(1-\tilde \E\exp(I_1+I_2))+(\tilde\E\exp(2I_1+2I_2)-1).\label{eq:rrr}
\end{equs}
Also let us denote %$I_i=\sum_{k\in\N}I_i^k$ for $i=1,2$, as well as
\begin{equ}
I_{3}=\int_0^t\Big( \scal{n}^\gamma(W^n_s-W^{\delta,n}_s)\partial e^{n}_\eps(x+\hat W_{t-s})\Big)^2\,ds
\end{equ}
and make a couple of simple observations:

\begin{enumerate}[(a)]
\item By Novikov's condition, for all $\lambda \in \R$, if $\tilde\E\exp (\frac{\lambda^2}{2} I_{3})<\infty$, then  $\tilde \E \exp(\lambda I_1-\frac{\lambda^2}{2}I_{3})=1$.
\item For any $\kappa_1,>0$, the sequence $W:=(W^n)_{n\in\Z}$ is a Gaussian random variable in a separable space with the norm
\begin{equ}
\,[f]:=\sup_{n\in\Z} \scal{n}^{-\kappa_1}\|f^n\|_{C^{\frac{1}{2}-\kappa_1}}.
\end{equ}
By Fernique's theorem there exists an $\alpha>0$ such that $\E \exp (\alpha[W]^2)<\infty$.
\item For any $\kappa_2>0$,
\begin{equs}
|e^n_{\eps}(x)|&=(\rho_\eps(x-\cdot),e^n)_{L^2}=\big((1-\Delta)^{\frac{1+\kappa_2}{2}}\rho_\eps(x-\cdot),(1-\Delta)^{-\frac{1+\kappa_2}{2}}e^n\big)_{L^2}
\\
&\lesssim \scal{n}^{-1-\kappa_2}\eps^{-1-\kappa_2},
\end{equs}
and similarly, $|\d e^n_{\eps}(x)|\lesssim \scal{n}^{-1-\kappa_2}\eps^{-2-\kappa_2}$, $|\d^2e^n_{\eps}(x)|\lesssim \scal{n}^{-1-\kappa_2}\eps^{-3-\kappa_2}$.
\end{enumerate}
Taking $\kappa_1\in(0,\frac{1}{2})$ arbitrary and $\kappa_2>\kappa_1+\gamma$, (c) implies
\begin{equs}
|I_2|\lesssim [W]\delta^{\frac{1}{2}-\kappa_1}\eps^{-3-\kappa_2},\qquad
|I_3|\lesssim[W]^2 \delta^{1-2\kappa_1}\eps^{-6-2\kappa_2}.
\end{equs}
%In particular, by (b) we have 
%Therefore, if $\eps$ is sufficiently small and $\delta\leq \eps^{15}$, we are in the framework of Proposition \ref{prop:easy exponential} to bound exponential moments of $I_2, I_3$.
Now fix $A$ such that $A(1-2\kappa_1)>6+2\kappa_2$ and in the sequel we assume $\delta\leq \eps^{A}$. Therefore by (b), for all small enough $\eps>0$, $\exp(8 I_3)<\infty$. 
%Now we write
%\begin{equ}\label{eq:something in the middle}
%\tilde\E|1-\exp(I_1+I_2)|^2= 2(1-\tilde \E\exp(I_1+I_2))+(\tilde\E\exp(2I_1+2I_2)-1).
%\end{equ}
%The two terms are treated in the same way, so we only detail the first one. Furthermore, just as in Proposition \ref{prop:easy exponential}, we only detail the bound without the absolute values, the rest easily follows from Jensen's inequality.
So for any $\lambda\in[-2,2]$ we can use (a) to write
\begin{equs}
\tilde \E\exp(\lambda I_1+\lambda I_2)&=\tilde \E\exp(\lambda I_1-\lambda^2 I_3+\lambda^2 I_3+\lambda I_2)
\\
&\leq\big(\tilde\E\exp(2\lambda I_1-2\lambda^2 I_3)\big)^{\frac{1}{2}}\big(\tilde\E\exp(2\lambda I_2+2\lambda^2 I_3)\big)^{\frac{1}{2}}
\\
&=\big(\tilde\E\exp(2\lambda I_2+2\lambda^2 I_3)\big)^{\frac{1}{2}}
\\
%&\leq\big(1+N\delta^{1/6}\eps^{-2}\big)\leq 1+N'\delta^{1/6}\eps^{-2},
&\leq\big(\tilde\E\exp(\delta^{\frac{1}{2}-\kappa_1}\eps^{-3-\kappa_2}K(1+[W]^2)\big)^{\frac{1}{2}}.
%\\
%&\leq \big(1+K\delta^{1/3}\eps^{-4}\tilde\E\exp(\alpha(1+[W]_1^2))\big)^{1/2}.
\end{equs}
Here and below $K$ denotes some constant independent of $\delta$ and $\eps$ that may change from line to line.
Using the elementary inequality $a\leq 1+b a^{\frac{1}{b}}$, which holds for all $a\geq 0$, $b\in(0,1]$, and the exponential integrability of $[W]^2$, we can conclude 
\begin{equ}\label{eq:exp1}
\tilde \E\exp(\lambda I_1+\lambda I_2)\leq 1+K\delta^{\frac{1}{2}-\kappa_1}\eps^{-3-\kappa_2}.
\end{equ}
%We can conclude that
%\begin{equ}
%\tilde \E\exp(2I_1+2I_2)-1\lesssim \delta^{1/3}\eps^{-4}.
%\end{equ}
%Similarly, one gets
%\begin{equ}
%\tilde \E\exp(-I_1-I_2)-1\lesssim\delta^{1/3}\eps^{-4}.
%\end{equ}
Therefore by Jensen's inequality
%To spell out: -EX=-1/(1/EX)\leq -1/E(1/X)
\begin{equs}
-\tilde \E\exp(\lambda I_1+\lambda I_2)\leq-\big(\tilde \E\exp(-\lambda I_1-\lambda I_2)\big)^{-1}&\leq -(1+K\delta^{\frac{1}{2}-\kappa_1}\eps^{-3-\kappa_2})^{-1}
\\&\leq -1+K\delta^{\frac{1}{2}-\kappa_1}\eps^{-3-\kappa_2}.\label{eq:exp2}
\end{equs}
Using \eqref{eq:exp1} with $\lambda=2$ and \eqref{eq:exp2} with $\lambda=-1$ in \eqref{eq:rrr} we get
\begin{equ}
\E| v^{(\eps)}_t(x)-v^{(\delta,\eps)}_t(x)|\lesssim \big(\delta^{\frac{1}{2}-\kappa_1}\eps^{-3-\kappa_2}\big)^{\frac{1}{2}},
\end{equ}
which, coupled with Corollary \ref{cor:strong-approx} finishes the proof.
\end{proof}
\begin{remark}
Since all the estimates are quantitative, a (tiny) convergence rate can also be obtained from the proof.
\end{remark}

%This equation fits once again in the variational framework, but this time with the $(\omega,t)$-dependent diffusion coefficient 
%\begin{equ}
%B^n(\omega,t) v=n^\gamma\big(\sigma(u(\omega,t))e^n-\sigma(u(\omega,t)-v)e^n_\eps\big).
%\end{equ}
%We wish to bound $z$. To do so, as before, it suffices to check A-1, A-3, and A-4 from \cite{Neelima}. The verification of A-1 and A-4 is exactly as before, 

\textbf{Acknowledgments}

MG is funded by the European Union (ERC, SPDE, 101117125). Views and opinions expressed
are however those of the author(s) only and do not necessarily reflect those of the European Union
or the European Research Council Executive Agency. Neither the European Union nor the granting
authority can be held responsible for them.

%\begin{remark}
%For Cole-Hopf to be really useful, one should prove that the solution of \eqref{eq:SHE} is strictly positive. For $\gamma=0$ this is the result of Muller doi:10.1080/17442509108833738. But I think this point might be tangential, one could work on always on the level of \eqref{eq:SHE}.
%\end{remark}

\bibliographystyle{Martin}
\bibliography{ito-frac} 

\newcommand{\etalchar}[1]{$^{#1}$}
\begin{thebibliography}{HHL{\etalchar{+}}17}
\expandafter\ifx\csname url\endcsname\relax
  \def\url#1{\texttt{#1}}\fi
\expandafter\ifx\csname urlprefix\endcsname\relax\def\urlprefix{URL }\fi
\expandafter\ifx\csname href\endcsname\relax
  \def\href#1#2{#2}\fi
\expandafter\ifx\csname burlalt\endcsname\relax
  \def\burlalt#1#2{\href{#2}{\texttt{#1}}}\fi

\bibitem[BGN24]{BGN}
\textsc{Y.~Bruned}, \textsc{M.~Gerencs{\'e}r}, and \textsc{U.~Nadeem}.
\newblock {Quasi-generalised KPZ equation} (2024).
\newblock \burlalt{arXiv:2401.13620}{http://arxiv.org/abs/2401.13620}.
\newblock
  \burlalt{doi:10.48550/ARXIV.2401.13620}{http://dx.doi.org/10.48550/ARXIV.2401.13620}.

\bibitem[BMSS95]{BMS}
\textsc{V.~Bally}, \textsc{A.~Millet}, and \textsc{M.~Sanz-Sol{\'e}}.
\newblock Approximation and support theorem in h{\"o}lder norm for parabolic
  stochastic partial differential equations.
\newblock \emph{The Annals of Probability} \textbf{23}, no.~1, (1995),
  178--222.
\newblock
  \burlalt{doi:10.1214/aop/1176988383}{http://dx.doi.org/10.1214/aop/1176988383}.

\bibitem[DGG21]{DGG}
\textsc{K.~Dareiotis}, \textsc{M.~Gerencs{\'e}r}, and \textsc{B.~Gess}.
\newblock Porous media equations with multiplicative space-time white noise.
\newblock \emph{Annales de l'Institut Henri Poincar{\'e}, Probabilit{\'e}s et
  Statistiques} \textbf{57}, no.~4(2021).
\newblock
  \burlalt{doi:10.1214/20-aihp1139}{http://dx.doi.org/10.1214/20-aihp1139}.

\bibitem[DPZ92]{DPZ}
\textsc{G.~Da~Prato} and \textsc{J.~Zabczyk}.
\newblock \emph{Stochastic equations in infinite dimensions}, vol.~44 of
  \emph{Encyclopedia of Mathematics and its Applications}.
\newblock Cambridge University Press, Cambridge, 1992.
\newblock
  \burlalt{doi:10.1017/CBO9780511666223}{http://dx.doi.org/10.1017/CBO9780511666223}.

\bibitem[GT19]{Gu2019}
\textsc{Y.~Gu} and \textsc{L.-C. Tsai}.
\newblock Another look into the wong–zakai theorem for stochastic heat
  equation.
\newblock \emph{The Annals of Applied Probability} \textbf{29}, no.~5(2019).
\newblock
  \burlalt{doi:10.1214/19-aap1474}{http://dx.doi.org/10.1214/19-aap1474}.

\bibitem[GT24]{Fabio}
\textsc{M.~Gerencs{\'e}r} and \textsc{F.~Toninelli}.
\newblock Weak coupling of the {KPZ} equation with rougher than white noise.
\newblock \emph{In preparation} (2024).

\bibitem[Hai24]{Martin}
\textsc{M.~Hairer}.
\newblock Renormalisation in the presence of variance blowup (2024).
\newblock \burlalt{arXiv:2401.10868}{http://arxiv.org/abs/2401.10868}.
\newblock
  \burlalt{doi:10.48550/ARXIV.2401.10868}{http://dx.doi.org/10.48550/ARXIV.2401.10868}.

\bibitem[HHL{\etalchar{+}}17]{long}
\textsc{Y.~Hu}, \textsc{J.~Huang}, \textsc{K.~L{\^e}}, \textsc{D.~Nualart}, and
  \textsc{S.~Tindel}.
\newblock Stochastic heat equation with rough dependence in space.
\newblock \emph{The Annals of Probability} \textbf{45}, no.~6B, (2017),
  4561--4616.
\newblock
  \burlalt{doi:10.1214/16-aop1172}{http://dx.doi.org/10.1214/16-aop1172}.

\bibitem[Hos16]{Hoshino}
\textsc{M.~Hoshino}.
\newblock {KPZ} equation with fractional derivatives of white noise.
\newblock \emph{Stochastics and Partial Differential Equations: Analysis and
  Computations} \textbf{4}, no.~4, (2016), 827--890.
\newblock
  \burlalt{doi:10.1007/s40072-016-0078-x}{http://dx.doi.org/10.1007/s40072-016-0078-x}.

\bibitem[HP15]{HAIRER2015}
\textsc{M.~Hairer} and \textsc{{\'E}.~Pardoux}.
\newblock A wong-zakai theorem for stochastic pdes.
\newblock \emph{Journal of the Mathematical Society of Japan} \textbf{67},
  no.~4(2015).
\newblock
  \burlalt{doi:10.2969/jmsj/06741551}{http://dx.doi.org/10.2969/jmsj/06741551}.

\bibitem[HW22]{Hu2022}
\textsc{Y.~Hu} and \textsc{X.~Wang}.
\newblock Stochastic heat equation with general rough noise.
\newblock \emph{Annales de l'Institut Henri Poincar{\'e}, Probabilit{\'e}s et
  Statistiques} \textbf{58}, no.~1, (2022), 379--423.
\newblock
  \burlalt{doi:10.1214/21-aihp1161}{http://dx.doi.org/10.1214/21-aihp1161}.

\bibitem[KR81]{KR_SEE}
\textsc{N.~V. Krylov} and \textsc{B.~L. Rozovskii}.
\newblock Stochastic evolution equations.
\newblock \emph{Journal of Soviet Mathematics} \textbf{16}, no.~4, (1981),
  1233--1277.
\newblock
  \burlalt{doi:10.1007/bf01084893}{http://dx.doi.org/10.1007/bf01084893}.

\bibitem[Kry99]{K_Lp}
\textsc{N.~V. Krylov}.
\newblock An analytic approach to {SPDE}s.
\newblock In \emph{Stochastic partial differential equations: six
  perspectives}, vol.~64 of \emph{Math. Surveys Monogr.},  185--242. Amer.
  Math. Soc., Providence, RI, 1999.
\newblock
  \burlalt{doi:10.1090/surv/064/05}{http://dx.doi.org/10.1090/surv/064/05}.

\bibitem[LR15]{LR}
\textsc{W.~Liu} and \textsc{M.~R{\"o}ckner}.
\newblock \emph{SPDEs with Locally Monotone Coefficients},  123--178.
\newblock Springer International Publishing, 2015.
\newblock
  \burlalt{doi:10.1007/978-3-319-22354-4}{http://dx.doi.org/10.1007/978-3-319-22354-4}.

\bibitem[N{\v{S}}19]{Neelima}
\textsc{Neelima} and \textsc{D.~{\v{S}}i{\v{s}}ka}.
\newblock {Coercivity condition for higher moment a priori estimates for
  nonlinear SPDEs and existence of a solution under local monotonicity}.
\newblock \emph{Stochastics} \textbf{92}, no.~5, (2019), 684--715.
\newblock
  \burlalt{doi:10.1080/17442508.2019.1650043}{http://dx.doi.org/10.1080/17442508.2019.1650043}.

\bibitem[Par75]{Pardoux}
\textsc{{\'E}.~Pardoux}.
\newblock \emph{{Equations aux deriv\'ees partielles stochastiques non
  lineaires monotones. Etude de solutions fortes de type Ito}}.
\newblock Ph.D. thesis, Univ. Paris Sud., 1975.

\bibitem[Tri98]{Triebel}
\textsc{H.~Triebel}.
\newblock \emph{Interpolation Theory Function Spaces Differential Operators}.
\newblock John Wiley {\&} Sons Inc, 1998.

\bibitem[Wal86]{Walsh}
\textsc{J.~B. Walsh}.
\newblock An introduction to stochastic partial differential equations.
\newblock In \textsc{P.~L. Hennequin}, ed., \emph{{\'E}cole d'{\'E}t{\'e} de
  Probabilit{\'e}s de Saint Flour XIV - 1984},  265--439. Springer Berlin
  Heidelberg, Berlin, Heidelberg, 1986.

\end{thebibliography}
\end{document}